\documentclass[12pt,reqno]{amsart}
\usepackage[a4paper,margin=29mm]{geometry}
\usepackage[T1]{fontenc}
\usepackage{lmodern}
\usepackage{amsmath,amssymb,amsthm,mathtools,mathrsfs}
\usepackage{microtype}
\usepackage{needspace}
\usepackage{tikz}
\usepackage{tikz-cd}
\usetikzlibrary{arrows.meta,calc}
\usepackage{enumitem}
\usepackage{xcolor}
\usepackage[
  colorlinks=true,
  linkcolor=red,
  citecolor=blue,
  urlcolor=blue
]{hyperref}

\newtheorem{maintheorem}{Theorem}

\newtheorem{proposition}{Proposition}[section]
\newtheorem{lemma}[proposition]{Lemma}
\newtheorem{theorem}[proposition]{Theorem}
\newtheorem{definition}[proposition]{Definition}
\newtheorem{corollary}[proposition]{Corollary}

\theoremstyle{remark}

\newcommand{\C}{\mathbb C}
\newcommand{\R}{\mathbb R}
\newcommand{\Q}{\mathbb Q}
\newcommand{\Z}{\mathbb Z}
\newcommand{\M}{\mathcal M}
\newcommand{\T}{\mathcal T}
\newcommand{\GL}{\operatorname{GL}_2^+(\R)}
\newcommand{\SL}{\operatorname{SL}_2(\R)}
\newcommand{\PSL}{\operatorname{PSL}_2(\R)}
\DeclareMathOperator{\Area}{Area}
\DeclareMathOperator{\covol}{covol}

\DeclareMathOperator{\PD}{PD}

\DeclareMathOperator{\Aut}{Aut}

\setlist[enumerate,1]{label=(\roman*),ref=(\roman*),leftmargin=2.3em,
  itemsep=.35em,topsep=.45em}
\numberwithin{equation}{section}
\allowdisplaybreaks[1]

\title[Finiteness of algebraic WP
  totally geodesic curves in moduli spaces]{Finiteness of algebraic Weil--Petersson\\
  totally geodesic curves in pointed moduli spaces}
  \author[Tianzhi Hu]{Tianzhi Hu}
\address{
School of Mathematics and Statistics, Wuhan University, Wuhan, Hubei 430072, P.R. China
}
\email{hutianzhi@whu.edu.cn}

\author[Kang Zuo]{Kang Zuo}
\address{
School of Mathematics and Statistics, Wuhan University, Wuhan, Hubei 430072, P.R. China;
Shanghai Institute for Mathematics and Interdisciplinary Sciences (SIMIS), Shanghai, P.R. China;
Institut f\"ur Mathematik, Johannes Gutenberg-Universit\"at Mainz, Mainz 55099, Germany
}
\email{zuok@uni-mainz.de}
\date{}

\begin{document}

\begin{abstract}
For every pair $(g,n)$ with $2g-2+n>0$, we prove that $\M_{g,n}$
contains only finitely many closed algebraic curves
that are totally geodesic for the Weil--Petersson metric.
We show that every such curve is an arithmetic quadratic Teichm\"uller
curve and has equal heights for all marked maximal cylinders in every periodic
direction. Although there are infinitely many arithmetic
Teichm\"uller curves, we can bound the
degree of the associated torus covering by the crucial equal-height property. This leads to the desired finiteness statement.

We also prove that quadratic Teichmüller curves are characterized by the existence of a holomorphic splitting of the logarithmic cotangent Kodaira–Spencer map. This result is the analogue of Möller and Viehweg’s characterization of Kobayashi geodesics in $\mathcal A_g$.
\end{abstract}
\maketitle

\section{Introduction}

The moduli space $\M_{g,n}$ of genus $g$ smooth curves with
$n$ marked points carries two natural metrics: the
Teichm\"uller metric and the Weil--Petersson (WP) metric.
The former is a complete Finsler metric, which coincides
with the Kobayashi metric \cite{Royden}, whereas the latter
is an incomplete K\"ahler metric \cite{MasurWP}.
The study of totally geodesic subvarieties with respect
to these two metrics is a natural problem in the geometry
of moduli spaces.

Algebraic totally geodesic subvarieties for the
Teichm\"uller metric have been extensively studied.
McMullen established the rigidity of Teichm\"uller curves
\cite{McMullen}. Nevertheless, infinitely many 
Teichm\"uller curves exist in a fixed moduli space,
as illustrated by the arithmetic examples arising from
square-tiled surfaces \cite{HubertLelievre}.
In contrast, Wright proved that, for each fixed $(g,n)$,
there are only finitely many totally geodesic subvarieties
of complex dimension at least two
\cite{WrightGeodesic}.
Higher-dimensional examples not arising from covering
constructions were constructed by McMullen, Mukamel, and
Wright \cite{MMW}.

For the Weil--Petersson metric, the theory of algebraic
totally geodesic subvarieties is considerably less
developed. Recently, Serv\'an established the local rigidity of WP
totally geodesic subvarieties of $\M_{g,n}$ \cite{Servan}.
Serv\'an also asked whether every such subvariety arises from a covering construction. In contrast to the Teichmüller setting, neither a general finiteness theorem nor an example not arising from a covering construction has been established.

In this paper, we study algebraic curves in $\M_{g,n}$
that are totally geodesic with respect to the
Weil--Petersson metric. We prove the following finiteness theorem.

\begin{maintheorem}\label{thm:finiteness}
For each pair of integers $g,n\geq0$ with $2g-2+n>0$, there are only
finitely many closed algebraic WP totally geodesic curves in $\M_{g,n}$.
\end{maintheorem}

For $\M_{0,3}$, the statement is empty. In each of the one-dimensional cases
$\M_{0,4}$ and $\M_{1,1}$ the only closed irreducible algebraic curve
is the whole moduli space. It is natural to ask whether this finiteness holds for higher-dimensional algebraic WP totally geodesic subvarieties in $\M_{g,n}$.

The key to our finiteness theorem is a structural Theorem~\ref{thm:structure} for closed algebraic WP totally geodesic curves.
For such a curve $\mathcal C$, total geodesy yields a
holomorphic splitting $\beta$ of the cotangent
Kodaira--Spencer map; see \eqref{eq:beta}. For any (pointed) curve $[(X,P)]\in\mathcal C$, the image of $\beta$ determines a quadratic
differential $q\in H^0(X,K_X^2(P))$, and we prove that $\mathcal C$ is the
associated Teichm\"uller curve. Moreover, this curve is
arithmetic, and all cylinders in every periodic direction
of $q$ have equal heights. These properties, together
with the compactness theorem of Eskin--Mirzakhani--
Mohammadi \cite[Theorem~2.3]{EMM}, lead to the finiteness
result.

We briefly recall the relevant terminology. Let
$[(X,P)]\in\M_{g,n}$, where $P=\{x_1,\ldots,x_n\}$, and
write the hyperbolic metric on $X\setminus P$ as
$\varrho(z)|dz|^2$. The WP metric is the K\"ahler metric
dual to the Hermitian pairing
\[
\langle q_1,q_2\rangle_{\mathrm{WP}}
=\int_{X\setminus P}
\frac{q_1\overline{q_2}}{\varrho(z)}\,dx\,dy
\]
on $H^0(X,K_X^2(P))$ \cite{MasurWP}.
A curve is WP totally geodesic if its tangent bundle is
preserved by the WP Levi--Civita connection.

We recall some standard terminology concerning quadratic
differentials and translation surfaces; see
\cite{Veech,MasurCylinders}.

Let $[(X,P)]\in\M_{g,n}$ and let $0\ne q\in H^0(X,K_X^2(P))$.
Thus $q$ is holomorphic away from $P$, has at most simple
poles at $P$, and satisfies $\int_X|q|<\infty$.
The triple $(X,P,q)$ is a \textbf{pointed quadratic
differential surface}. Away from the zeros of $q$ and
$P$, the local coordinates $z=\int\sqrt q$ give the
flat metric $|q|=|dz|^2$. The $\GL$-action on these
coordinates generates a Teichm\"uller disk.
If its image in $\M_{g,n}$ is an algebraic curve, we
call it a \textbf{quadratic Teichm\"uller curve}.

For nonsquare $q$, the \textbf{canonical orientation cover}
is the connected double cover $p:Y\to X$ on which
$p^*q=\omega^2$ for a holomorphic Abelian differential
$\omega$. If $q$ is a square, we take $Y=X$,
$p=\operatorname{id}$, and $\omega^2=q$.
The pair $(Y,\omega)$ is a \textbf{translation surface}:
away from the zeros $Z(\omega)$ of $\omega$, its local coordinates
$z=\int\omega$ have translation transition maps.
Throughout, unless otherwise stated, we fix the distinguished set
\begin{equation}\label{eq:distinguished}
\Sigma:=Z(\omega)\cup p^{-1}(P).
\end{equation}
We call the quadratic Teichm\"uller curve
\textbf{arithmetic} if, after applying an element of
$\GL$, the marked translation surface $(Y,\omega,\Sigma)$
admits a translation covering onto a flat torus
$\C/\Lambda$, ramified only over the origin, such that
every point of $\Sigma$ maps to the origin.
This strengthens the usual arithmeticity condition
\cite{GutkinJudge} by also requiring the marked points to map to the origin.

A \textbf{saddle connection} is a straight segment in the
flat metric joining zeros of $q$ or points of $P$,
with no such point in its interior.
A direction $\vartheta$ is called \textbf{periodic} if
the union of its saddle connections, the zeros of $q$,
and $P$ has complement consisting entirely of cylinders
of parallel closed trajectories. These cylinders are called \textbf{marked maximal cylinders}.
Each cylinder $C$ is isometric to
\[
(\R/L\Z)\times(0,h),
\]
where $L>0$ is its circumference and $h>0$ its
transverse height. A \textbf{core curve} $\gamma$ of $C$
is a simple closed geodesic in $C$,
representing a generator of $H_1(C,\Z)$.
Let $\delta$ be a straight geodesic segment perpendicular
to the closed trajectories, whose interior lies in $C$
and whose endpoints lie on the two boundary components
of the metric completion of $C$.
In the flat coordinate
$z=\int e^{-i\vartheta}\sqrt q=x+iy$,
defined modulo horizontal translations, the closed
trajectories are horizontal, and
\[
L=\left|\int_\gamma\sqrt q\right|,
\qquad
h=\left|\operatorname{Im}
\int_\delta e^{-i\vartheta}\sqrt q\right|.
\]
Thus $L$ is the flat length of a core curve, and $h$ is the transverse height of the cylinder. In a periodic direction, the cylinder boundaries consist of saddle connections meeting the zeros of $q$ or the marked points $P$.

We now state the structural theorem for algebraic
WP totally geodesic curves.

\begin{maintheorem}\label{thm:structure}
Let $\mathcal C\subset\M_{g,n}$ be a closed algebraic WP totally geodesic curve,
where $2g-2+n>0$.
Then:
\begin{enumerate}
\item\label{it:disk}
$\mathcal C$ is a quadratic Teichm\"uller curve (generated by some pointed quadratic surface).
\item\label{it:heights}
In every periodic direction of the above pointed quadratic surface,
all marked maximal cylinders have the same height.
\item\label{it:arithmetic}
$\mathcal C$ is an arithmetic quadratic Teichm\"uller curve.
\end{enumerate}
\end{maintheorem}
We state the above properties separately and in the order in which they
are proved. Using Theorem~\ref{thm:structure}\ref{it:arithmetic}, we answer the dimension-one case
of Serv\'an's question \cite[Question 1.5]{Servan}. More precisely,
Corollary~\ref{coro:WPcurve} shows that every such WP totally
geodesic curve arises from a covering construction.

In the course of proving Theorem~\ref{thm:structure}, we realized
that the argument in fact yields a new characterization of Teichm\"uller
curves. We first recall the
relevant background.

Viehweg and Zuo developed the theory of maximal
Higgs fields and used it to characterize special
families and Shimura curves through equality in
Arakelov inequalities \cite{VZ03,VZ04}.
Building on this, M\"oller characterized
Teichm\"uller curves generated by holomorphic Abelian differentials in $\M_g$ by the existence of
a rank-two Higgs subbundle with maximal Higgs field
in the weight-one variation of Hodge structures
\cite[Theorem~5.3]{Moller}.
M\"oller and Viehweg also characterized
Kobayashi geodesics in $\mathcal A_g$ by the splitting of the logarithmic cotangent map
\cite[Theorem~1.2]{MV10}.
Motivated by these results, we establish the following
characterization of quadratic Teichm\"uller curves
in $\M_{g,n}$ in terms of the splitting of the logarithmic
Kodaira--Spencer map.

We fix the notation here. Let $\mathcal C\subset\M_{g,n}$ be a closed irreducible
algebraic curve. After passing to a
finite cover of its normalization and
taking stable reduction, we obtain a stable pointed
family over a smooth projective curve $B$:
\[
f:(\mathcal X;S_1,\ldots,S_n)\longrightarrow B,
\]
and write $S=\sum_{j=1}^n S_j$ and let $D\subset B$ be the reduced divisor of singular fibers. We refer to the beginning
of Section~\ref{sec:wp} for the construction.
\begin{maintheorem}
\label{thm:log-characterization}
With the notation above, the following are equivalent:
\begin{enumerate}
\item\label{it:log-teich}
$\mathcal C$ is a quadratic Teichm\"uller curve.
\item\label{it:log-split}
The logarithmic cotangent Kodaira--Spencer map
\[
\kappa:
f_*\bigl(\omega_{\mathcal X/B}^{\otimes2}(S)\bigr)
\longrightarrow\Omega_B^1(\log D)
\]
admits a holomorphic right inverse
\begin{equation*}
\beta:\Omega_B^1(\log D)\longrightarrow
f_*\bigl(\omega_{\mathcal X/B}^{\otimes2}(S)\bigr)
\end{equation*}
satisfying $\kappa\circ\beta=\operatorname{id}$.
\end{enumerate}
\end{maintheorem}

The proof is organized as follows. In
Section~\ref{sec:wp}, WP total geodesicity yields
a holomorphic splitting $\beta$ of the Kodaira--Spencer map. Using plumbing coordinates
\cite{WolpertNodal} and the boundary asymptotics
of the WP metric
\cite{MasurWP,ObitsuWolpert,MelroseZhu},
we establish the logarithmic extension of $\beta$ and study its asymptotic behaviour. In Section~\ref{sec:hodge} and Appendix~\ref{app:algebraic-splitting}, we prove \ref{it:log-split}$\Rightarrow$\ref{it:log-teich} of Theorem~\ref{thm:log-characterization}
by realizing the pullback of the base curve's uniformizing
Higgs bundle as a rank-two Higgs factor of the weight-one
VHS arising from the orientation-cover family of the quadratic curves family.
The converse is proved directly in Section~\ref{sec:converse}. In Section~\ref{sec:degeneration}, we use cusp
monodromy and the Veech group \cite{Veech} to prove
that the absolute periods form a lattice and that
all cylinders in each periodic direction have
equal heights. In Section~\ref{sec:flat}, these
properties first yield a lattice of relative periods on the marked
orientation surface and hence prove the arithmeticity assertion of
Theorem~\ref{thm:structure}\ref{it:arithmetic}.
The compactness theorem of Eskin--Mirzakhani--Mohammadi
\cite{EMM} then gives a bound for the degree of the associated
torus coverings. This finally proves
Theorem~\ref{thm:finiteness}.

\vspace{.5cm}

\noindent\textbf{Acknowledgements:}
We would like to thank Xin L\"u for helpful suggestions and discussions on the proof of Appendix~\ref{app:algebraic-splitting}.

\vspace{.5cm}

\noindent\textbf{AI Declaration:} The authors proposed the research questions and some ideas of this paper. ChatGPT found the equal-height property as stated in Theorem~\ref{thm:structure}. AI tools were also used for proving intermediate results and polishing the writing. All
mathematical statements have been independently checked by the authors, who take full
responsibility for the content.

\section{The logarithmic extension of the WP splitting}\label{sec:wp}
Pass to a finite cover of $\M_{g,n}$ that carries a universal
family, and normalize a component lying above $\mathcal C$, denoted by $U$. The induced map $U\to \M_{g,n}$ is immersed by \cite[\S2.2, footnote~5]{Servan}.

Choose the stable-reduction base change to be unramified
over $U$; see \cite{DeligneMumford,Knudsen}, which yields a family
\[
f:(\mathcal X;S_1,\ldots,S_n)\longrightarrow B
\]
over the smooth projective completion $B$. Write $S=\sum_{j=1}^n S_j$, and let
$D\subset B$ be the reduced divisor of singular fibers,
with $U=B\setminus D$. The sections $S_1,\cdots,S_n$ are disjoint,
avoid the nodes, and remain marked on rational components.

\begin{definition}[Relative dualizing sheaf]\label{def:dualizing}
For the stable pointed family $f$, let $\omega_{\mathcal X/B}$ denote
the relative dualizing sheaf, characterized by
$f^!\mathcal O_B\simeq\omega_{\mathcal X/B}[1]$; see \cite{HartshorneDuality}. On a smooth fiber it restricts to the canonical bundle;
on a nodal fiber, its sections are differentials on the normalization
with at most simple poles at the preimages of the nodes and opposite
residues. 
\end{definition}

Consequently, $\omega_{\mathcal X/B}^{\otimes2}$ describes
stable quadratic differentials, with at most double poles at the nodes,
while $\omega_{\mathcal X/B}^{\otimes2}(S)$ also allows simple poles
along the marked sections.

Relative Serre duality identifies $f_*\bigl(\omega_{\mathcal X/B}^{\otimes2}(S)\bigr)$ with
$\bigl(R^1f_*\mathcal T_{\mathcal X/B}(-S)\bigr)^\vee$; in particular,
$f_*\bigl(\omega_{\mathcal X/B}^{\otimes2}(S)\bigr)$ is locally free of rank $3g-3+n$.

\begin{definition}[Plumbing coordinates and nodal coefficients]\label{def:plumbing}
Let $t$ be a local coordinate on $B$ with $(t=0)\in D$, and let
$x_i,y_i$ be coordinates on disjoint disks on the two branches of
the normalization at the $i$th node. A \textbf{plumbing chart} replaces
the node $x_i y_i=0$ by the annulus $x_i y_i=t_i$, where the
\textbf{plumbing coordinate} $t_i$ is a local smoothing parameter,
$t_i=0$ being the nodal fiber. Along the base curve,
\begin{equation}\label{eq:plumbing}
 t_i=t^{m_i}u_i(t),\qquad m_i\in\Z_{>0},\quad u_i(0)\ne0;
\end{equation}
$m_i$ is the order of approach to the boundary divisor. A
\textbf{stable quadratic differential} is a section of $\omega_{\mathcal X/B}^{\otimes2}(S)$. On the central fiber $t=0$, its
restriction to the two branches of the normalization has Laurent
expansions
\begin{align*}
 q&=\bigl(a_i+O(x_i)\bigr)\Bigl(\frac{dx_i}{x_i}\Bigr)^2
 \quad\text{on the $x_i$-branch};\\
 q&=\bigl(a_i+O(y_i)\bigr)\Bigl(\frac{dy_i}{y_i}\Bigr)^2
 \quad\text{on the $y_i$-branch},
\end{align*}
with the \textbf{same} coefficient $a_i$, as follows from the opposite
residue condition for sections of the dualizing sheaf, the sign
disappearing upon squaring. On each nearby smooth annulus
$x_iy_i=t_i\neq0$ these two descriptions glue into a single joint
expression
\begin{equation}\label{eq:laurent-smooth}
 q=F_i(x_i,y_i,t)\Bigl(\frac{dx_i}{x_i}\Bigr)^2,
\end{equation}
where $F_i$ is holomorphic in $(x_i,y_i,t)$ for
$|x_i|,|y_i|<\varepsilon$ and $|t|$ small, with
$F_i(0,0,0)=a_i$; this is valid because
$\frac{dx_i}{x_i}=-\frac{dy_i}{y_i}$ on the annulus, so
$\bigl(\frac{dx_i}{x_i}\bigr)^2=\bigl(\frac{dy_i}{y_i}\bigr)^2$ is a
single-valued quadratic generator. We call $a_i$ the \textbf{double-pole
coefficient}.
\end{definition}

\begin{definition}[Logarithmic Kodaira--Spencer map]\label{def:kappa}
Over $U$, define
\[
\kappa_U:f_*\bigl(\omega_{\mathcal X/B}^{\otimes2}(S)\bigr)|_U\longrightarrow\Omega_U^1,\qquad
\kappa_{U,b}(q)(v)
=\langle q,\mathrm{KS}_b(v)\rangle_b,
\]
where
$\mathrm{KS}_b:T_bU\to H^1(X_b,T_{X_b}(-P_b))$
is the pointed Kodaira--Spencer map and the pairing is Serre duality. By the logarithmic deformation theory of
stable curves, the map $\kappa_U$ extends uniquely to
\[
\kappa:f_*\bigl(\omega_{\mathcal X/B}^{\otimes2}(S)\bigr)\longrightarrow\Omega_B^1(\log D).
\]
\end{definition}

Because the WP metric is K\"ahler, its Levi--Civita connection
preserves the complex structure. Total geodesy therefore makes
$T\mathcal C$ parallel along the curve; its orthogonal complement is
parallel as well. Hence we have the splitting
\[
 \beta_U:\Omega_U^1\longrightarrow
 f_*\bigl(\omega_{\mathcal X/B}^{\otimes2}(S)\bigr)\big|_U,
 \qquad \kappa\circ\beta_U=\mathrm{id}_{\Omega^1_U},
\]
over $U$. The following proposition
extends $\beta_U$ across the boundary divisor $D$ and computes the
nodal behaviour of the extension.

\begin{proposition}\label{prop:splitting}
The splitting $\beta_U$ extends to a holomorphic splitting
\begin{align}\label{eq:beta}
 \beta:\Omega_B^1(\log D)
       \longrightarrow f_*\bigl(\omega_{\mathcal X/B}^{\otimes2}(S)\bigr).
\end{align}
With the conventions of Definition~\ref{def:plumbing}, the limiting
double-pole coefficients of $\beta(dt/t)$ satisfy
\begin{equation}\label{eq:nodal}
 a_i=c\,m_i^{-2}
\end{equation}
for a common $c\in\C^\times$. The limiting differential is nonzero on
every irreducible component of the stable fiber. Moreover, as $t\to0$, we have
\begin{equation}\label{eq:cusp-L1}
 \int_{X_t}|\beta(dt)|\asymp |t|\log(1/|t|).
\end{equation}
\end{proposition}

\begin{proof}
If $D=\varnothing$, then $U=B$ and \eqref{eq:beta} is merely the original
interior splitting $\beta=\beta_U$; in this case no boundary behaviour
needs to be established. So we focus on the case $D\ne\emptyset$.

\emph{The nodal coefficients.}
Let $X_0:=f^{-1}(0)$ have nodes indexed by $i=1,\ldots,k$. The
\textbf{coefficient map}
\begin{align}\label{eq:evmap}
 \nu:f_*\bigl(\omega_{\mathcal X/B}^{\otimes2}(S)\bigr)_0
 \longrightarrow\C^k,
 \qquad
 \nu(q)=(a_1(q),\ldots,a_k(q)),
\end{align}
is linear and surjective \cite{WolpertNodal}, so that
\[
 0\longrightarrow \ker\nu\longrightarrow f_*\bigl(\omega_{\mathcal X/B}^{\otimes2}(S)\bigr)_0
 \xrightarrow{\ \nu\ }\C^k\longrightarrow0 .
\]
Choose sections $q_1,\ldots,q_k$ of $f_*\bigl(\omega_{\mathcal X/B}^{\otimes2}(S)\bigr)$ near $t=0$ with
$\nu(q_i)=e_i$, and complete them to a local frame
$(q_1,\ldots,q_{3g-3+n})$ by any basis of $\ker\nu$; this is a
\textbf{normalized nodal frame}. 

For $i=1,\cdots,k$, put $\ell_i=-\log|t_i|$ and $\ell=-\log|t|$, so that along the
base
\begin{equation}\label{eq:ell}
 \ell_i=m_i\ell+O(1),\qquad \ell\to\infty .
\end{equation}
Let $G$ be the Weil--Petersson matrix for the frame chosen above, i.e. $\|\sum b_jq_j\|^2=\overline b^{\,t}Gb$. The normal--tangential block
asymptotics of \cite{MasurWP,ObitsuWolpert,MelroseZhu} in this frame
give
\begin{equation}\label{eq:gram}
 G=\begin{pmatrix}
 C_0\operatorname{diag}(\ell_i^3)+O(\ell^2)&O(1)\\
 O(1)&G_0+o(1)
 \end{pmatrix},
 \qquad C_0>0,\quad G_0>0,
\end{equation}
understood in the joint degeneration \eqref{eq:ell}.

For each $t\ne0$, the lift $\beta(dt/t)$ is the unique section $q$ of
$f_*(\omega_{\mathcal X/B}^{\otimes2}(S))$ such that
\[
 \langle t\partial_t,q\rangle=1
 \qquad\text{and}\qquad
 \|q\|_{\mathrm{WP}}\ \text{is minimal}.
\]
In logarithmic plumbing coordinates the tangent vector is
\[
 t\partial_t=\sum_i(m_i+O(t))t_i\partial_{t_i}
                +O(t)\partial_{\mathrm{thick}} .
\]Here $\partial_{\mathrm{thick}}$ denotes a tangent direction in
the remaining, nondegenerating moduli parameters.

Also, we may assume $q=\sum\limits_{i=1}^{3g-3+n} b_i q_i.$ We use the following normalization of Wolpert's plumbing-pairing
formula \cite{WolpertNodal}: if the principal part at the $i$th node is
$a_i(dx_i/x_i)^2$, then the contribution of $t_i\partial_{t_i}$ to
the pairing is $-\pi a_i+o(1)$. Thus
\[
\langle t\partial_t,q\rangle=-\pi\sum_{i=1}^k m_ib_i+o(1)=1.
\]
Minimizing the WP norm $\bar b^{\,t}Gb$ subject to this constraint and using \eqref{eq:gram} gives
\begin{equation}\label{eq:lift-coords}
 b_i=-\frac{1}{\pi}\,\frac{m_i/\ell_i^3}{\sum_{j=1}^k m_j^2/\ell_j^3}\bigl(1+O(\ell^{-1})\bigr),\qquad 1\leq i\leq k.
\end{equation}
Finally, since
$\ell_i=m_i\ell+O(1)$, \eqref{eq:lift-coords} converges to
\[
 -\frac{1}{\pi}\,\frac{m_i^{-2}}{\sum_{j=1}^k m_j^{-1}},
\]
i.e.\ the limiting nodal coefficients satisfy $a_i=c\,m_i^{-2}$ for
a common $c\in\C^\times$, as claimed.

\emph{Extension and the $L^1$ estimate.} From \eqref{eq:lift-coords} we have, as $t\to0$, each $b_i$ tends to a nonzero constant for $1\le i\le k$, while the tangential coordinates satisfy
\[
 b_\alpha=O(1),
 \qquad k<\alpha\le 3g-3+n,
\]
since the off-diagonal blocks of $G^{-1}$ are $O(\ell^{-3})$ and
$m^tG^{-1}m\asymp\ell^{-3}$. Thus in the holomorphic frame
$(q_1,\ldots,q_{3g-3+n})$ all coordinate functions of
$\beta(dt/t)$ are bounded on the punctured disk; by Riemann's
removable singularity theorem they extend holomorphically to $t=0$,
so $\beta(dt/t)$ extends to a holomorphic section of
$f_*(\omega_{\mathcal X/B}^{\otimes2}(S))$.

The identity
$\kappa\circ\beta=\mathrm{id}$, which holds on
the dense open set $U$ and both sides of which are holomorphic
logarithmic bundle maps, then holds on all of $B$.

Since every irreducible component of the connected nodal
fiber $X_0$ contains at least one node, and since the limiting
double-pole coefficient is nonzero at every node, the limiting stable quadratic differential
is nonzero on every component of $X_0$.

On the $i$th plumbing annulus $\{x_i y_i=t_i\}$,
the area element $|dx_i/x_i|^2$ has area form $d\log|x_i|\,d\theta$, where $x_i=|x_i|e^{\sqrt{-1}\theta}$. The annulus is described by
\[
 \log\frac1\varepsilon\le \log\frac1{|x_i|}\le \ell_i-\log\frac1\varepsilon,
\]
and the principal part of $\beta(dt/t)$ there is
$a_i(t)\,(dx_i/x_i)^2$ with $a_i(t)\to a_i\neq0$. Hence
\begin{equation}\label{eq:annulus-L1}
 \int_{\mathrm{ann}_i}\Bigl|a_i\Bigl(\frac{dx_i}{x_i}\Bigr)^2\Bigr|
 =2\pi|a_i|\,\ell_i+O(1).
\end{equation}
The remaining
part is uniformly bounded by \eqref{eq:laurent-smooth} and contributes
$O(1)$.

At a marked section $z=0$ (that is, in local coordinates $(z,t)$
on the total space, the image of a marked section $S_j:B\to\mathcal X$
is given by $z=0$ in each nearby fiber), $\beta(dt/t)$ has a bounded simple-pole
coefficient, so locally $|q|\le C/|z|$, and the polar area element
cancels the pole:
\begin{equation}\label{eq:mark-L1}
 \int_{|z|<\varepsilon}|q|
 \;\le\; C\int_0^\varepsilon\frac{1}{r}\,r\,dr
 \;=\; C\varepsilon .
\end{equation}
Thus the marked punctures contribute no logarithmic divergence.

Summing \eqref{eq:annulus-L1} over the $k$ nodes and using
$\ell_i\asymp\ell$,
\[
 \int_{X_t}\bigl|\beta(dt/t)\bigr|
 =\sum_{i=1}^k 2\pi|a_i|\ell_i+O(1)
 \;\asymp\;\ell,
\]
which is \eqref{eq:cusp-L1}. Finally, $\beta$ is algebraic by GAGA.
\end{proof}

\section{A characterization of the Teichm\"uller curve}\label{sec:hodge}

We first prove the implication
\ref{it:log-split}$\Rightarrow$\ref{it:log-teich} of
Theorem~\ref{thm:log-characterization} under the abstract
logarithmic splitting hypothesis, without using WP total
geodesicity or WP-specific asymptotics. Together with
Proposition~\ref{prop:splitting}, this proves
Theorem~\ref{thm:structure}\ref{it:disk}. And then in the last subsection we prove the converse implication.

\subsection{The orientation cover and the sub-VHS}\label{sec:cover}
We use the quadratic differential line determined
by $\beta$ to construct a canonical orientation cover
and, in Lemma~\ref{lem:factor}, a rank-two subvariation
of Hodge structures. Finally, we show that the associated
period metric descends to $U$ and coincides with its
Poincar\'e metric. Comparison with the Teichm\"uller
metric completes the proof.

We construct the orientation-cover family from the
logarithmic splitting $\beta$. Choose a finite map
$\pi_1:B_1\to B$ of even degree. Since every line bundle of even
degree on a smooth projective complex curve has a square root, we may
choose $\mathcal L_1$ on $B_1$ with
\[
 \mathcal L_1^{\otimes2}\simeq
 \pi_1^*\Omega_B^1(\log D).
\]

Set
\[
 f_1:(\mathcal X_1;S_1)
 :=(\mathcal X;S)\times_B B_1\longrightarrow B_1.
\]
Let $b_1:\mathcal X_1\to\mathcal X$ be the projection, so that
$S_1=b_1^*S$. Since the fibers of $f$ are nodal and hence
Gorenstein, the relative dualizing sheaf commutes with arbitrary
base change \cite[Tag~0E6R]{Stacks}. Consequently,
\[
 \omega_{\mathcal X_1/B_1}\simeq b_1^*\omega_{\mathcal X/B},
 \qquad
 b_1^*\bigl(\omega_{\mathcal X/B}^{\otimes2}(S)\bigr)
 \simeq\omega_{\mathcal X_1/B_1}^{\otimes2}(S_1).
\]
As $\pi_1$ is a finite morphism between smooth curves, it is flat.
The flat base-change theorem \cite[Tag~02KH]{Stacks} therefore yields
\[
 \pi_1^*f_*\bigl(\omega_{\mathcal X/B}^{\otimes2}(S)\bigr)
 \simeq (f_1)_*\bigl(\omega_{\mathcal X_1/B_1}^{\otimes2}(S_1)\bigr).
\]
Using this isomorphism and
$\mathcal L_1^{\otimes2}\simeq\pi_1^*\Omega_B^1(\log D)$,
the pullbacks of $\beta$ and $\kappa$ give
\[
 \beta_1:\mathcal L_1^{\otimes2}
 \longrightarrow
 (f_1)_*\bigl(\omega_{\mathcal X_1/B_1}^{\otimes2}
 (S_1)\bigr),\qquad
 \kappa_1\beta_1=\operatorname{id}.
\]
Here $\kappa_1$ denotes the base change of $\kappa$, with target
$\mathcal L_1^{\otimes2}$.
By adjunction, $\beta_1$ corresponds to a section
\[
 q_1\in H^0\!\left(\mathcal X_1,
 \omega_{\mathcal X_1/B_1}^{\otimes2}(S_1)
 \otimes f_1^*\mathcal L_1^{-2}\right).
\]
The marked points require some care: $q_1$ is permitted
to have simple poles there. Put
\[
 \mathcal M_1:=
 \omega_{\mathcal X_1/B_1}(S_1)
 \otimes f_1^*\mathcal L_1^{-1}.
\]

Then $q_1$ is a section of
$\mathcal M_1^{\otimes2}(-S_1)$, hence also of
$\mathcal M_1^{\otimes2}$. It induces an
$\mathcal O_{\mathcal X_1}$-algebra structure on
\[
 \mathcal A_1
 :=\mathcal O_{\mathcal X_1}\oplus\mathcal M_1^{-1},
\]
where multiplication on
$\mathcal M_1^{-1}\otimes\mathcal M_1^{-1}
=\mathcal M_1^{-2}$
is given by the homomorphism
\[
 q_1:\mathcal M_1^{-2}
 \longrightarrow\mathcal O_{\mathcal X_1}.
\]
The associated flat double cover is
\[
 \varpi_1:\mathcal W_1
 :=\operatorname{Spec}_{\mathcal X_1}(\mathcal A_1)
 \longrightarrow\mathcal X_1;
\]
see \cite[\S~3.5]{EsnaultViehweg} for further discussion
of cyclic covers and their properties. The defining quadratic algebra has a canonical tautological section
\[
 \vartheta_1\in H^0(\mathcal W_1,\varpi_1^*\mathcal M_1),
 \qquad \vartheta_1^{\otimes2}=\varpi_1^*q_1.
\]
For a local generator $a_1$ of $\mathcal L_1$, this gives the
square-root form
\begin{equation}\label{eq:tautological-root}
 w_{1,a_1}:=\vartheta_1\cdot(f_1\circ\varpi_1)^*a_1,
 \qquad
 w_{1,a_1}^{\otimes2}=\varpi_1^*\beta_1(a_1^{\otimes2}),
\end{equation}
where $w_{1,a_1}$ is a section of
$\varpi_1^*\omega_{\mathcal X_1/B_1}(S_1)$.

Consider the geometric generic fiber of the cover
$\varpi_1:\mathcal W_1\to\mathcal X_1$.
If $q_1$ is not a square there, its normalization is a connected
orientation double cover. If $q_1$ is a square, its normalization
consists of two copies of the geometric generic curve, corresponding
to the two choices of square root.

We now choose the finite base change and the covering family according
to the above two cases. In either
case, once the finite map $\pi_2:B'\to B_1$ specified below has been
chosen, put
\[
 \pi=\pi_1\circ\pi_2,\qquad
 \mathcal L=\pi_2^*\mathcal L_1,\qquad
 f':(\mathcal X';S')=(\mathcal X_1;S_1)\times_{B_1}B'
 \longrightarrow B',
\]
and form the base-changed double cover and its normalization:
\[
 \mathcal Y_1=\mathcal W_1\times_{B_1}B',\qquad
 \mathcal Y_2=(\mathcal Y_1)^\nu.
\]

\emph{Case 1: The nonsquare case.}
Suppose that $q_1$ is not a square on the geometric generic fiber.
Choose $\pi_2:B'\to B_1$ so that the normalized generic double cover admits
a semistable model over $B'$ on which the deck involution determined
by $q_1$ extends. Take such an equivariant normal semistable model
\[
 g:\mathcal Y\longrightarrow B',
\]
denote its involution by $\sigma$, and set
\[
 \tau=2,\qquad E:=(g_*\omega_{\mathcal Y/B'})^-.
\]
The generic cover is connected, and $\mathcal Y$ is birational
to $\mathcal Y_2$ over $B'$.

\emph{Case 2: The square case.}
Suppose that $q_1$ is a square on the geometric generic fiber.
Choose $\pi_2:B'\to B_1$ so that a square root
of the geometric generic quadratic differential is defined
over $\C(B')$. Since $\mathcal X'$ is
normal, the normalization of the base-changed double cover is
\[
 \mathcal Y_2\simeq\mathcal X'\sqcup\mathcal X'.
\]
Choose the component corresponding to one square root and put
\[
 g:\mathcal Y:=\mathcal X'\longrightarrow B',\qquad \tau=1,\qquad
 E:=g_*\omega_{\mathcal Y/B'}.
\]
The chosen component maps naturally into $\mathcal Y_2$.

In either case, $\mathcal L^{\otimes2}\simeq
\pi^*\Omega_B^1(\log D)$. The base-change gives
\[
 \beta':\mathcal L^{\otimes2}\longrightarrow
 f'_*\omega_{\mathcal X'/B'}^{\otimes2}(S'),\qquad
 \kappa':f'_*\omega_{\mathcal X'/B'}^{\otimes2}(S')
 \longrightarrow\mathcal L^{\otimes2},\qquad
 \kappa'\beta'=\operatorname{id}.
\]
The two constructions
are summarized by the same commutative diagram:
\begin{equation}\label{eq:diagram}
\begin{tikzcd}[column sep=small,row sep=small]
\mathcal Y \arrow[d,"g"'] \arrow[r,dashed] &
\mathcal Y_2 \arrow[d] \arrow[r,"\nu"] &
\mathcal Y_1 \arrow[d] \arrow[r,"\varpi"] &
\mathcal X' \arrow[d,"f'"] \arrow[r] &
\mathcal X_1 \arrow[d,"f_1"] \arrow[r] &
\mathcal X \arrow[d,"f"] \\
B' \arrow[r,equal] & B' \arrow[r,equal] & B'
\arrow[r,equal] & B' \arrow[r,"\pi_2"'] &
B_1 \arrow[r,"\pi_1"'] & B.
\end{tikzcd}
\end{equation}
Here $\varpi$ is the base change of $\varpi_1$, and $\nu$ is the
normalization map. In either case, the composite defines the
induced rational map $p:\mathcal Y\dashrightarrow\mathcal X'$,
whose generic degree is $\tau$.

Let $D'\subset B'$ be the reduced divisor consisting of
$\pi^{-1}(D)$, the interior ramification points of $\pi$, the
parameters where $g$ has singular fibers, and those where the
zero or pole multiplicities of the quadratic differentials determined
by $\beta'$ change.

The family $g:\mathcal Y\to B'$ is projective and semistable,
with normal total space and smooth generic fiber. Fix a dense open subset
$V\subset B'\setminus D'$ on which $p$ restricts to a finite
morphism $p_V:\mathcal Y_V\to\mathcal X'_V$. The square-root
form in \eqref{eq:tautological-root}, after base change,
normalization, and the natural pullback of relative differentials,
gives for every local generator $a$ of $\mathcal L|_V$ a
holomorphic one-form $w_a$ on $\mathcal Y_V$ satisfying
\begin{align}\label{eq:wa}
 w_a^{\otimes2}=p_V^*\beta'(a^{\otimes2}).
\end{align}
In the nonsquare case, $\sigma^*w_a=-w_a$. 

Over the smooth locus $B'\setminus D'$, the weight-one
polarized variation of Hodge structure of $g$ is
$R^1g_*\C$; in the nonsquare case we take its
$(-1)$-eigensummand under $\sigma$. The Higgs bundle of
its Deligne canonical extension \cite{Deligne,Schmid} is
\[
 E\oplus E^\vee,\qquad
 \theta:E\longrightarrow E^\vee\otimes\Omega_{B'}^1(\log D').
\]
Here $E=E^{1,0}$ (identified with the direct image specified in the above covering construction), and the polarization identifies $E^{0,1}$ with $E^\vee$.

We call the points of $D'\setminus\pi^{-1}(D)$
the \textbf{auxiliary interior points}. The differential of
$\pi$ induces
\[
\iota:=d\pi:
\mathcal L^{\otimes2}
\longrightarrow\Omega_{B'}^1(\log D'),
\]
which vanishes precisely at the auxiliary interior
points. Consider the rank-two Higgs bundle
\begin{align}\label{eq:pullbackuni}
\left(
\mathcal L\oplus\mathcal L^{-1},
\begin{pmatrix}
0&0\\
\iota&0
\end{pmatrix}
\right).
\end{align}
Its projectivization is the pullback of the
uniformizing Higgs bundle of $U$. 

However, we cannot expect the Higgs bundle \eqref{eq:pullbackuni} to be maximal due to the existence of the auxiliary interior points. Instead, we will establish the \textbf{compatibility} of \eqref{eq:pullbackuni}
with the Higgs bundle $(E\oplus E^\vee,\theta)$ of the
weight-one VHS of $\mathcal Y/B'$. This compatibility
is the main ingredient.

\begin{lemma}\label{lem:factor}
With the notation above, the following hold.
\begin{enumerate}
\item\label{it:existence}
There exist a saturated injection
$\lambda:\mathcal L\hookrightarrow E$
and a map $r:E\to\mathcal L$ satisfying
$r\lambda=\operatorname{id}_{\mathcal L}$.

\item\label{it:higgs}
The Higgs bundle $(E\oplus E^\vee,\theta)$ contains
the saturated degree-zero rank-two Higgs subbundle
\[
\lambda(\mathcal L)\oplus
r^\vee(\mathcal L^{-1}).
\]
More precisely, for every local section $\eta$ of
$E$ and local section $a$ of $\mathcal L$, the Higgs pairing satisfies
\begin{equation}\label{eq:theta}
\theta(\lambda(a),\eta)
=\tau\,\iota\bigl(a\otimes r(\eta)\bigr).
\end{equation}
In particular, the induced Higgs field on this
rank-two subbundle is
\[
\begin{pmatrix}
0&0\\
\tau\iota&0
\end{pmatrix}.
\]
This subbundle underlies an irreducible complex
rank-two flat subvariation $\mathbb V$, whose
Hodge bundles are isomorphic to $\mathcal L$ and
$\mathcal L^{-1}$ and whose polarization has
signature $(1,1)$.
\end{enumerate}
\end{lemma}

\begin{proof}
\emph{\ref{it:existence}.}
Work over the open set $V$ fixed above, and put $E_V=E|_V$.
For a local generator $a$ of $\mathcal L|_V$, let $w_a$ be the
holomorphic relative one-form \eqref{eq:wa} obtained from the tautological
root \eqref{eq:tautological-root}. These forms give a well-defined bundle map
\[
 \lambda_V:\mathcal L|_V\longrightarrow E_V,
 \qquad \lambda_V(a)=w_a.
\]

For a local section $\xi$ of $E_V$, define $r_V(\xi)$ by
\begin{equation}\label{eq:retraction}
 a\otimes r_V(\xi)
 =\kappa'\!\left(
  \frac1\tau\operatorname{Tr}_{p_V}(w_a\xi)\right).
\end{equation}
In the square case the trace is the identity. In the nonsquare
case both $w_a$ and $\xi$ are anti-invariant with respect to the involution, so their product
is invariant and descends as a rational quadratic differential.
It has at most simple poles along the marked sections. Indeed,
at a branch point over a zero of odd order $k\ge1$, or over a
marked simple pole ($k=-1$), take an \'etale-local coordinate
$z=u^2$ and write
\[
 \beta'(a^{\otimes2})=z^k c(z)(dz)^2,
 \qquad
 w_a=2u^{k+1}c_1(u^2)\,du,\qquad
 \xi=c_2(u^2)\,du,
\]
where $c_1^2=c$ and $c_1,c_2$ are regular. Then
\[
 w_a\xi=\tfrac12 z^{(k-1)/2}c_1(z)c_2(z)(dz)^2,
\]
which is regular for $k\ge1$ and has at most a simple pole
for $k=-1$. Away from ramification, the descent is regular
because $p_V$ is \'etale. Thus the normalized trace in
\eqref{eq:retraction} belongs to
$f'_*\omega_{\mathcal X'/B'}^{\otimes2}(S')|_V$.
Moreover, replacing $a$ by $ha$ multiplies this trace by $h$,
so \eqref{eq:retraction} is independent of the generator $a$
and defines an algebraic bundle map
$r_V:E_V\to\mathcal L|_V$.
Finally, the defining square-root identity yields
\[
 a\otimes r_V\lambda_V(a)
 =\kappa'\!\left(\frac1\tau
   \operatorname{Tr}_{p_V}(w_a^{\otimes2})\right)
 =\kappa'\beta'(a^{\otimes2})=a^{\otimes2}.
\]
Therefore $r_V\lambda_V=\operatorname{id}_{\mathcal L|_V}$. The extension of $\lambda_V$ and $r_V$ to $B'$ is proved in Lemma~\ref{lem:algebraic-splitting} in
Appendix~\ref{app:algebraic-splitting}. Their splitting identity extends by density; hence
$\lambda$ is a saturated split injection.

\emph{\ref{it:higgs}.}
Fix $b'\in V$, let $p_{b'}:Y_{b'}\to X'_{b'}$ be the fiberwise covering map, and choose a local generator $a$ of $\mathcal L$.
Put $w_a=\lambda_V(a)$. For a tangent vector
$v\in T_{b'}V$, let $\mu_v$ represent the Kodaira--Spencer
class of $g$ on $Y_{b'}$. The Higgs pairing is
\[
\theta_v(w_a,\xi)
=\int_{Y_{b'}}\mu_v\,w_a\xi,
\qquad \xi\in E_{b'}.
\]
Away from the ramification points of $p_{b'}$, this pairing
descends via the trace map. We check that moving branch
points contribute no additional terms.

At a ramification point lying over an odd-order zero
$z_0(s)$ of $\beta'(a^{\otimes2})$, of order $k\ge1$,
choose a local parameter $s$ on $V$ and relative coordinates
$z$ on $\mathcal X'$ and $u$ on $\mathcal Y$ such that
\[
z-z_0(s)=u^2.
\]
Then $w_a=O(u^{k+1})\,du$. Differentiating the covering
relation shows that the motion of $z_0(s)$ contributes
a meromorphic correction vector field
$\delta=O(u^{-1})\partial_u$. Since $\xi$ is holomorphic,
\[
\delta\lrcorner(w_a\xi)=O(u^k)\,du,
\qquad
\operatorname{Res}_{u=0}
\bigl(\delta\lrcorner(w_a\xi)\bigr)=0.
\]
Thus such moving branch points contribute no residues.
At a branch point lying over a marked simple pole,
a relative coordinate centered at the marked section
makes its branch value fixed. Its motion is already
included in the pointed Kodaira--Spencer class, so no
additional correction is needed.

Consequently, over $V$ the Higgs pairing agrees with
the pointed Kodaira--Spencer pairing after trace:
\[
\theta_v(w_a,\xi)=
\bigl[\iota\bigl(\kappa'(
\operatorname{Tr}_{p_{b'}}(w_a\xi))\bigr)\bigr](v).
\]
By \eqref{eq:retraction}, this equals
$\tau[\iota(a\otimes r_V(\xi))](v)$, proving
\eqref{eq:theta} on $V$. Since $\lambda$ and $r$
extend holomorphically, the identity extends across
$B'\setminus V$ as an identity of logarithmic
Higgs-bundle maps.

By \eqref{eq:theta} and
$r\lambda=\operatorname{id}$,
\[
 \lambda(\mathcal L)\oplus
 r^\vee(\mathcal L^{-1})
\]
is a saturated rank-two Higgs subbundle:
the first line maps into the second tensored
with $\Omega_{B'}^1(\log D')$, and the Higgs
field vanishes on the second. Its degree is zero.

The ambient polarized VHS has a polystable
tame logarithmic Higgs bundle of parabolic
degree zero \cite{Simpson90,Simpson92};
unipotent monodromy gives zero parabolic
weights. The equality case of Higgs
semistability makes this degree-zero saturated
subbundle a harmonic orthogonal direct summand,
hence a flat subvariation. Its construction
respects the Hodge grading, and its polarization
has signature $(1,1)$.

Semistability applied to the Higgs line
$r^\vee(\mathcal L^{-1})$ gives
$\deg\mathcal L\ge0$. If equality held,
this line would split off as a flat factor,
forcing the nilpotent Higgs field of the
rank-two summand to vanish, contrary to the
generic nonvanishing of $\iota$.
Thus $\deg\mathcal L>0$. Since the only
generic Higgs-invariant line of this nonzero
nilpotent rank-two system is
$r^\vee(\mathcal L^{-1})$, of negative degree,
the system is stable. Simpson's correspondence
therefore gives the irreducible complex
rank-two flat subvariation $\mathbb V$.
This proves \ref{it:higgs}.
\end{proof}

\subsection{Comparison with the Teichm\"uller metric}\label{sec:periodmetric}

On each smooth fiber $Y_{b'},\ b'\in B'$ we use
\[
 \|\omega\|^2
 =\frac{i}{2}\int_{Y_{b'}}\omega\wedge\overline\omega
 =\int_{Y_{b'}}|\omega|^2,
\]
and the dual norm on $E_{b'}^\vee$. 

\begin{lemma}\label{lem:metric}
The pullback of the Teichm\"uller metric to $U$
coincides with its curvature-$(-4)$ Poincar\'e metric.
\end{lemma}

\begin{proof}
Write the Poincar\'e metric locally as
$\rho_U(t)|dt|$, with $\rho_U(v)$ denoting the
norm of any vector $v$. By Lemma~\ref{lem:factor}, the Higgs
model of $\mathbb V$ is
$\mathcal L\oplus\mathcal L^{-1}$ with lower-left
entry $\tau\iota$. Its uniformizing harmonic metric
is $h_0\oplus h_0^{-1}$, where
\[
 h_0(a,a)=\frac{\tau|c|}{\rho_U(t)},
 \qquad a^{\otimes2}=c\,\pi^*dt.
\]
In frames with $a^{\otimes2}=\pi^*(dt/t)$ at cusps,
$h_0(a,a)\asymp\log(1/|t|)$; in frames with
$a^{\otimes2}=\pi^*dt$ at auxiliary interior points,
it is bounded above and below. Thus the filtered
extensions have zero parabolic weights, matching
the canonical extension of $\mathbb V$.
The Hodge metric restricts harmonically to $\mathbb V$;
uniqueness of adapted metrics \cite{Simpson90}
identifies their projective metrics.

Fix $b'\in V$, put $t=\pi(b')$, choose a local
generator $a$ of $\mathcal L$, and write
$w_a=\lambda(a)|_{Y_{b'}}$. For $v\in T_tU$ and
its lift $\widetilde v\in T_{b'}B'$, the comparison
gives
\[
 \rho_U(v)
 =\frac{\|\theta_{\widetilde v}(w_a)\|}{\|w_a\|}.
\]
Here the numerator uses the dual Hodge norm.
For any Beltrami representative $\mu$ of $v$ on
$X_t\simeq X'_{b'}$, the preceding Higgs-pairing
calculation gives, for $\eta\in E_{b'}$,
\begin{equation}\label{eq:pairing-bound}
 |\theta_{\widetilde v}(w_a,\eta)|
=\frac{1}{\tau}\left|
   \int_{X'_{b'}}\mu\,
   \operatorname{Tr}_{p_{b'}}(w_a\eta)
   \right|
 \leq\|\mu\|_\infty\|w_a\|\|\eta\|.
\end{equation}
Taking the dual norm and the infimum over $\mu$
yields
\begin{equation}\label{eq:upper}
 \rho_U(v)\leq\|v\|_{\mathrm{Teich}}.
\end{equation}
Schwarz--Pick gives the reverse inequality:
\begin{equation}\label{eq:metric-equality}
 \|v\|_{\mathrm{Teich}}\leq\rho_U(v).
\end{equation}
The equality extends across the finite exceptional
set by continuity.
\end{proof}

\begin{proof}[Proof of
Theorem~\ref{thm:log-characterization},
\ref{it:log-split}$\Rightarrow$\ref{it:log-teich}]
By Lemma~\ref{lem:metric}, the lifted map
$\widetilde U\to\T_{g,n}$ is a holomorphic
infinitesimal isometry. The holomorphic-disk equality
theorem \cite[Theorem~5 and \S8]{EKK}, equivalently
\cite[Theorem~2.1]{Antonakoudis}, identifies its image
with a Teichm\"uller disk. Its moduli image is
$\mathcal C$, so $\mathcal C$ is a quadratic
Teichm\"uller curve.

We show that $\mathbb V$ is real, as needed in
Section~\ref{sec:rational}. Use the notation above
at a generic $b'$, with $v\ne0$. The metric equality in Lemma~\ref{lem:metric} gives
\[
 \|\theta_{\widetilde v}(w_a)\|
 =\|v\|_{\mathrm{Teich}}\,\|w_a\|.
\]
Thus we have the equality in the above Cauchy--Schwarz inequality, which implies that $\theta_{\widetilde v}(w_a)$ is a nonzero scalar
multiple of $\overline{w_a}$. By \eqref{eq:theta},
this is $r^\vee(\mathcal L^{-1})_{b'}$, viewed
inside $E^{0,1}_{b'}$. Thus
\[
 r^\vee(\mathcal L^{-1})_{b'}
 =\overline{\lambda(\mathcal L)_{b'}},
 \qquad
 \mathbb V_{b'}=\overline{\mathbb V_{b'}}.
\]
Flatness and connectedness of $B'\setminus D'$
give $\mathbb V=\overline{\mathbb V}$ globally.
Thus $\mathbb V$ is a real
rank-two local system.

Finally, put
$q_a=\beta'(a^{\otimes2})|_{X'_{b'}}$, so that
$p_{b'}^*q_a=w_a^{\otimes2}$.
Equality in \eqref{eq:pairing-bound} gives
\[
 \mu=\zeta\,\frac{\overline{q_a}}{|q_a|}
 \quad\text{almost everywhere},
 \qquad |\zeta|=\|v\|_{\mathrm{Teich}}>0.
\]
Hence $q_a$ spans the tautological quadratic line.
Consequently $\operatorname{im}\beta'$ is its
pullback to $V$, and $\operatorname{im}\beta$ is
the tautological quadratic line on $U$.
\end{proof}

As a direct consequence of this result and Proposition~\ref{prop:splitting}, we obtain Theorem~\ref{thm:structure}\ref{it:disk}.

\subsection{The logarithmic splitting on a quadratic
Teichm\"uller curve}\label{sec:converse}
Readers interested only in the proofs of Theorems~\ref{thm:finiteness}
and \ref{thm:structure} may skip this section, as it is not used in
the later arguments.

In this section, we prove the implication \ref{it:log-teich}$\Rightarrow$\ref{it:log-split} in Theorem~\ref{thm:log-characterization}. We first give the following extension lemma.

\begin{lemma}\label{lem:polylog-extension}
Let $f:(\mathcal X;S_1,\ldots,S_n)\to\Delta$ be a stable
pointed family, smooth over $\Delta^*$. Suppose that $q_t$ is
a holomorphic section of
$f_*(\omega_{\mathcal X/\Delta}^{\otimes2}(S))$
over $\Delta^*$ satisfying
\[
\int_{X_t}|q_t|
=O\left(\left(\log1/|t|\right)^N\right)
\]
for some $N\geq0$ and $|t|\ll 1$. Then $q_t$ extends holomorphically
over $0$ as a stable quadratic differential.
\end{lemma}

\begin{proof}
Use the normalized nodal frame
$q_1,\ldots,q_r$, where $r=3g-3+n$,
constructed from the coefficient map in
\eqref{eq:evmap}. Write
\[
q_t=\sum_{j=1}^r b_j(t)q_j(t)
\]
over $\Delta^*$.

Choose a compact subset $K_0$ of the smooth
unmarked locus of $X_0$ whose interior meets every
irreducible component, and let $K_t\subset X_t$
be its deformation to nearby fibers. Since a stable
quadratic differential vanishing on $K_0$ vanishes
identically, the $L^1$-norm on $K_0$ defines a norm
on the space of stable quadratic differentials.
By continuity and finite dimensionality, there is
a constant $C>0$ such that
\[
\sum_{j=1}^r |b_j(t)|
\leq C\int_{K_t}|q_t|
\leq C\int_{X_t}|q_t|
\]
for all sufficiently small $t\neq0$.

The assumption therefore gives $b_j(t)=O\left(\left(\log1/|t|\right)^N\right).$ Hence $t b_j(t)\to0$ as $t\to0$, and the removable
singularity theorem shows that every $b_j$
extends holomorphically across $0$.
Thus $q_t$ extends as a stable quadratic differential.
\end{proof}

\begin{proof}[Proof of
Theorem~\ref{thm:log-characterization},
\ref{it:log-teich}$\Rightarrow$\ref{it:log-split}]
Let $(X,P,q)$ generate the Teichm\"uller curve,
and let $p:Y\to X$ be its orientation cover, with
$p^*q=w^2$. Put $\tau=2$ in the nonsquare case
and $\tau=1$, $p=\operatorname{id}$, in the square case.

Write $w=\alpha+\sqrt{-1}\,\nu$, where $\alpha,\nu$
are real closed forms. In the standard upper-half-plane
parametrization of the Teichm\"uller disk, we have
\begin{equation*}
w_\zeta=\alpha+\zeta\nu,\qquad
p_\zeta^*q_\zeta=w_\zeta^2.
\end{equation*}
Put $A=\tau^{-1}\int_Y\alpha\wedge\nu>0$.
The standard affine description \cite{Veech,Moller}
gives
\begin{equation}\label{eq:tautological-area}
\int_{X_\zeta}|q_\zeta|
=A\operatorname{Im}\zeta.
\end{equation}
Moreover, the infinitesimal deformation class is represented by
\[
\mu_{\partial_\zeta}
=-\frac{1}{\zeta-\overline\zeta}
\frac{\overline{q_\zeta}}{|q_\zeta|}.
\]
Thus the Kodaira--Spencer pairing gives
\[
\kappa(q_\zeta)=c\,d\zeta,
\qquad c=\frac{\sqrt{-1}A}{2}\neq0.
\]
Consequently, on $\mathbb H$ we may define
\begin{equation}\label{eq:defbetaU}
\widetilde\beta(d\zeta)=c^{-1}q_\zeta.
\end{equation}

We show that this map descends to $U$.
Write $U=\Gamma\backslash\mathbb H$.
For $\gamma\in\Gamma$, let
$F_{\gamma,\zeta}:X_\zeta\to X_{\gamma\zeta}$
be the induced isomorphism of pointed curves.
Since it preserves the quadratic direction defining
the disk, we have
\[
F_{\gamma,\zeta}^*q_{\gamma\zeta}
=h_\gamma(\zeta)q_\zeta
\]
for some nonzero scalar $h_\gamma(\zeta)$.
Naturality of the Kodaira--Spencer map gives
\[
c\,h_\gamma(\zeta)d\zeta
=\gamma^*(c\,d\zeta)
=c\,\gamma'(\zeta)d\zeta.
\]
Hence
\begin{equation}\label{eq:quadratic-equivariance}
F_{\gamma,\zeta}^*q_{\gamma\zeta}
=\gamma'(\zeta)q_\zeta.
\end{equation}
Thus $\widetilde\beta$ is equivariant and descends
to a holomorphic right inverse $\beta_U$ over $U$.

It remains to extend $\beta_U$ across the cusps.
At a cusp, choose an adapted coordinate $\zeta$
such that the parabolic generator acts by
$\gamma(\zeta)=\zeta+T$, with $T>0$, and put
\begin{equation*}
t=\exp(2\pi\sqrt{-1}\,\zeta/T).
\end{equation*}
Since $\gamma'(\zeta)=1$,
\eqref{eq:quadratic-equivariance} shows that
$q_\zeta$ is single-valued on the punctured
$t$-disk. By \eqref{eq:tautological-area},
\begin{equation*}
\int_{X_t}|q_\zeta|
=\frac{AT}{2\pi}\log\frac1{|t|}.
\end{equation*}
Lemma~\ref{lem:polylog-extension} therefore extends
$q_\zeta$ holomorphically as a stable quadratic
differential.

Since $\kappa$ extends logarithmically, the identity
\[
\kappa(q_\zeta)
=\frac{AT}{4\pi}\frac{dt}{t}
\]
continues across $t=0$. Hence
\[
\beta\left(\frac{dt}{t}\right)
=\frac{4\pi}{AT}q_\zeta
\]
defines a holomorphic right inverse near the cusp.
These local extensions agree with $\beta_U$
and therefore glue to the required splitting
on $B$.
\end{proof}

\section{Rationality and equal-height property}
\label{sec:degeneration}
From this section onward, $\mathcal C$ is again assumed to be
WP totally geodesic and $\beta$ is the WP orthogonal splitting of
Section~\ref{sec:wp}. The nodal formula \eqref{eq:nodal} used below
belongs to that additional hypothesis.

We now use the Teichm\"uller disk obtained in
Theorem~\ref{thm:structure}\ref{it:disk}. Since every
Teichm\"uller curve is noncompact \cite{Veech,MasurCylinders},
the boundary divisor $D$ is nonempty and gives rise to a genuine degeneration.
Its monodromy will prove the absolute-period statement in
Lemma~\ref{lem:absolute-periods}, and the nodal calculation of
Proposition~\ref{prop:splitting} will then prove the equal-height
assertion of Theorem~\ref{thm:structure}\ref{it:heights}.

\subsection{The absolute periods}\label{sec:rational}

\begin{lemma}\label{lem:absolute-periods}
Let $\mathcal C\subset\M_{g,n}$ be an algebraic WP totally geodesic
curve, where $2g-2+n>0$, and let $(X,P,q)$ generate its
Teichm\"uller disk. Write $p^*q=w^2$ on the canonical
orientation cover $p:Y\to X$; in the square case take $Y=X$
and $p=\operatorname{id}$. Then the group of absolute periods
\[
 \left\{\int_\gamma w:\gamma\in H_1(Y,\Z)\right\}
\]
is a rank-two lattice in $\C$.
\end{lemma}

\begin{proof}
\emph{A nonzero integral class.} By Theorem~\ref{thm:structure}\ref{it:disk} and the fact that every Teichm\"uller curve is noncompact \cite{Veech,MasurCylinders}, the curve $\mathcal C$ has $D\ne\emptyset$. Choose a point $0\in D$ and perform the auxiliary
base changes and equivariant semistable reduction used
in Section~\ref{sec:cover}. 

After the finite base changes making the monodromy unipotent,
let $T$ denote the monodromy on $H^1(Y_t,\mathbb Z)$ around $0$,
and put $N:=\log T$. Let $w_0$ be the limiting value of the section $w_t$ in the Deligne canonical extension of the
Hodge bundle. We write $Q_{\mathrm{cup}}$ for the cup-product
pairing
\[Q_{\mathrm{cup}}(\alpha,\beta):=\int_{Y_t}\alpha\wedge\beta\]
on $H^1(Y_t,\mathbb C)$. We choose the
Poincar\'e-duality sign so that
$Q_{\mathrm{cup}}(w_0,\mathrm{PD}(\delta))=\int_\delta w_0$,
where $\delta\in H_1$ and $\mathrm{PD}$ denotes Poincar\'e duality.

Write the original plumbing parameters as
$t_i=t^{m_i}u_i(t)$, with $u_i(0)\ne0$; here $m_i$
is the total exponent. In \eqref{eq:nodal}, $c$ denotes the common
coefficient for the logarithmic generator $dt/t$ on
the final base, including the rescaling induced by the
base changes. Choose $\nu\in\C^\times$ with $\nu^2=c$.

Let $X_0$ be the stable fiber over $0$, and let
$\mathcal Y\to B'\to B$ be the composed family defined in Section~\ref{sec:cover}. Denote its central
fiber by $Y_0$.
Near each original node, the limiting quadratic
differential has a nonzero double-pole coefficient,
so its square root has nonzero residue. Thus for $t\ne0$, the cover
$p_t:Y_t\to X_t$ is unramified over each
plumbing annulus corresponding to an original node. 
Above the $i$th original nodal collar, the orientation
cover has $\tau$ lifts, indexed by $j=1,\ldots,\tau$.
With compatible orientations of their vanishing
cycles $\delta_{ij}$, the limiting residues are
\[
 \operatorname{Res}_{ij}(w_0)
 =\epsilon_{ij}\frac{\nu}{m_i},
 \qquad \epsilon_{ij}\in\{\pm1\}.
\]
The sign comes from the
ambiguity of the choice of square root $w_0=\pm\sqrt{q_0}$ at each
lift of the node when $\tau=2$.

Semistable reduction may replace an original node
by a chain of exceptional components. The smoothing
multiplicities along the chain sum to $m_i$, while
the vanishing cycles are homologous and their
residues agree up to the chosen orientations.
Thus the total monodromy contribution of each
lifted chain is determined by the original
smoothing exponent $m_i$ and residue
$\epsilon_{ij}\nu/m_i$.

The remaining nodes arise over smooth or marked
points of $X_0$. Choose corresponding small disks
$\mathbb B_t\subset X_t$ around these points on the nearby fibers.
Since $\beta(dt/t)$ extends holomorphically as a stable quadratic differential,
\[
 \int_{p_t^{-1}(\mathbb B_t)}|w_t|^2
  =\tau\int_{\mathbb B_t}|q_t|=O(1).
\]
A nonzero limiting residue at such a node would
force logarithmic divergence of this integral
along its plumbing collar. Hence all these
additional nodes have zero residues.

Consequently, the Picard--Lefschetz formula
\cite{Clemens} gives
\begin{equation}\label{eq:monodromy-vector}
 Nw_0=(2\pi \sqrt{-1}\nu)v,
\end{equation}
where
\[
 v:=\sum_{i=1}^k\sum_{j=1}^{\tau}
 \epsilon_{ij}\PD(\delta_{ij})
 \in H^1(Y,\Z),
\]
and $\PD(\delta_{ij})$ denotes the Poincar\'e
dual of $[\delta_{ij}]\in H_1(Y,\Z)$. Then
\[
 Q_{\mathrm{cup}}(w_0,v)
       =2\pi \sqrt{-1}\nu\,\tau\sum_i m_i^{-1}\ne0.
\]
Thus $v$ is a nonzero integral class.

\emph{Rationality of the tautological plane.}
Let $P_{\mathbb V}$ denote the flat projection onto the rank-two
flat summand $\mathbb V$. Since
both summands are monodromy invariant, this projection commutes with
the monodromy and hence with $N$. It extends to the canonical
extension. Since $P_{\mathbb V}w_0=w_0$
in the limiting fiber, applying $P_{\mathbb V}$ to
\eqref{eq:monodromy-vector} gives
\[
 P_{\mathbb V}v=\frac{P_{\mathbb V}Nw_0}{2\pi \sqrt{-1}\nu}
 =\frac{NP_{\mathbb V}w_0}{2\pi \sqrt{-1}\nu}=v,
\]
so $P_{\mathbb V}v=v$ in the limiting fiber. Moreover $Nv=0$, as the
vanishing cycles are disjoint; hence $v$ extends to a flat section of
the extended local system. Since $P_{\mathbb V}$ is flat, it preserves
flat sections, and a flat section is determined by its value at one
point; therefore $P_{\mathbb V}v=v$ on the whole punctured disk. Thus
$v$ lies in the actual flat summand $\mathbb V$.

Every vector in the monodromy orbit of $v$ is rational and lies in
$\mathbb V$. By irreducibility
(Lemma~\ref{lem:factor}\ref{it:higgs}) their complex span is
all of $\mathbb V$; since these spanning vectors are rational, their
$\Q$-span also has dimension two. Hence $\mathbb V$ is defined over
$\Q$.

\emph{The actual period group.}
By Theorem~\ref{thm:structure}\ref{it:disk}, the real plane
underlying $\mathbb V$ is spanned by $\operatorname{Re}w$ and
$\operatorname{Im}w$. Choose a rational basis and clear denominators.
The absolute periods of $w$ then lie in a rank-two lattice in $\C$.
Their actual additive group also has rank two: rank at most one
would make $\operatorname{Re}w$ and $\operatorname{Im}w$ dependent,
contradicting the fact that the Hodge norm of $w$ is positive. Thus this group, which we
denote by $\Lambda_{\mathrm{abs}}$, is itself a rank-two lattice.
\end{proof}

\subsection{Equal cylinder heights}

Let $(X,P,q)$ be the pointed quadratic differential
surface generating $\mathcal C$, as in
Theorem~\ref{thm:structure}\ref{it:disk}.
Let $\Gamma\subset\PSL$ be its projective Veech group,
consisting of the projective derivatives of
orientation-preserving affine diffeomorphisms
preserving the marked points.
Since the flat coordinates of $q$ have transition
maps $z\mapsto\pm z+c$, these derivatives are
well defined up to sign.
As $\mathcal C$ is a Teichm\"uller curve,
$\Gamma$ is a lattice
\cite{Veech,MollerFiniteness}.

\begin{proof}[Proof of Theorem~\ref{thm:structure}\ref{it:heights}]
Let $\vartheta$ be any periodic direction. By Veech's
theorem \cite{Veech}, it determines a cusp of the
Teichm\"uller curve. Let $0\in D$ be the corresponding
boundary point. After a finite base change ramified
only over $D$, choose a local parameter $t$ at $0$.
Throughout, $m_i$ denotes the total plumbing exponent
for this parameter, and $\beta$ and \eqref{eq:nodal}
are understood on the same base. By the proof of
Theorem~\ref{thm:structure}\ref{it:disk},
$\beta(dt/t)$ spans the tautological quadratic line.

\emph{Cusp monodromy and affine multitwists.}
A positively oriented small loop around $t=0$
induces the cusp monodromy
$\rho_{\mathrm{cusp}}\in\operatorname{Mod}_{g,n}$,
where $\operatorname{Mod}_{g,n}$ is the mapping class
group preserving the ordered marked points.
Under the Teichm\"uller uniformization,
\[
\rho_{\mathrm{cusp}}=[F_A^e]
\]
for some $e>0$, where $A\in\Gamma$ is parabolic
in direction $\vartheta$ and $F_A$ is an affine
representative on the original pointed quadratic
surface
\cite[Proposition~2.4]{Veech};
cf.~\cite[p.~70]{MollerFiniteness}.
After a sufficiently large finite base change,
we may assume that $F_A^e$ preserves each marked
maximal cylinder and its boundary data.

Let $\gamma_i$ be the core curves of these cylinders.
They are \emph{essential}, meaning that none bounds
a disk containing at most one marked point, and
\emph{pairwise nonisotopic} in $X_t\setminus P_t$.
Indeed, flat Gauss--Bonnet excludes peripheral
cores, while an isotopy annulus between distinct
cores containing no marked points cannot contain
zeros of $q$, contradicting the maximality of the
cylinders. Along the corresponding Teichm\"uller
ray, the cylinder moduli tend to infinity, so the extremal length
\[
\operatorname{Ext}_{X_t\setminus P_t}(\gamma_i)
\leq \frac{1}{\operatorname{mod}(C_i(t))}
\longrightarrow 0.
\]
Thus the $\gamma_i$ degenerate to nodes in the
stable pointed degeneration
(cf.~\cite[p.~70]{MollerFiniteness}).

If the derivative of $F_A^e$ is $\left(\begin{smallmatrix}1&l\\0&1\end{smallmatrix}\right)$,
with $l>0$, then on the $i$th cylinder of circumference $L_i$ and
height $h_i$, the top boundary is displaced by $lh_i$ relative to
the bottom. Thus the cusp monodromy is a positive multitwist along
the $\gamma_i$, with twist exponents $lh_i/L_i$.

On the other hand, the monodromy of the stable pointed family is the
positive multitwist along its nodal vanishing cycles, with exponents
given by the plumbing orders. Uniqueness of multitwist
decompositions identifies these cycles with the $\gamma_i$, after
reindexing. Consequently,
\begin{equation}\label{eq:twists}
 \rho_{\mathrm{cusp}}=\prod_i T_{\gamma_i}^{m_i},
 \qquad m_i=\frac{l h_i}{L_i}\in\mathbb Z_{>0}.
\end{equation}
In particular, the corresponding plumbing parameters are
$t_i=t^{m_i}u_i(t)$, with $u_i(0)\ne0$
\cite[p.~70]{MollerFiniteness}.

\emph{Comparison of nodal coefficients.}
Put $q_t=\beta(dt/t)$, and let $w_t$ be a square root of its pullback
to the orientation cover. On the $i$th plumbing collar,
\eqref{eq:nodal} gives
\[
 q_t=\bigl(c\,m_i^{-2}+o(1)\bigr)
 \left(\frac{dx_i}{x_i}\right)^2
\]
along the middle circle $|x_i|=|t_i|^{1/2}$.
Let $\widetilde\gamma_i$ be a lift of the corresponding cylinder
core, homologous to a lift of this middle circle in the orientation
cover. Then
\[
 \left|\int_{\widetilde\gamma_i}w_t\right|^2
 \longrightarrow\frac{4\pi^2|c|}{m_i^2}.
\]
Along the chosen Teichm\"uller ray, the holonomies of the parallel
cylinder cores are multiplied by a common factor; the scalar
normalization of $q_t$ contributes another common factor. Thus, for
any $i,j$,
\[
 \frac{L_i^2}{L_j^2}
 =\frac{\left|\int_{\widetilde\gamma_i}w_t\right|^2}
        {\left|\int_{\widetilde\gamma_j}w_t\right|^2}
 \longrightarrow\frac{m_j^2}{m_i^2}.
\]
Consequently, $m_iL_i$ is independent of $i$, and
\eqref{eq:twists} shows that all $h_i$ are equal.
This proves Theorem~\ref{thm:structure}\ref{it:heights}.
\end{proof}

\section{Relative periods, arithmeticity, and finiteness}\label{sec:flat}

We first construct the canonical torus covering from the relative
periods and use it to prove
Theorem~\ref{thm:structure}\ref{it:arithmetic}.
We then bound its degree in Theorem~\ref{thm:flat} and deduce the finiteness Theorem~\ref{thm:finiteness}.

In this section $(Y,\omega,\Sigma)$ denotes a compact translation
surface with distinguished points; points of
$\Sigma\setminus Z(\omega)$ are \textbf{regular markings}: they carry no cone singularity, but cylinders are required to avoid
them, so each such point borders the cylinder decomposition. The absolute and relative period groups are
\[
 \Lambda_{\mathrm{abs}}
   =\left\{\int_\gamma\omega:\gamma\in H_1(Y,\Z)\right\},
 \qquad
 \Lambda_{\mathrm{rel}}
   =\left\{\int_\gamma\omega:\gamma\in H_1(Y,\Sigma;\Z)\right\};
\]
cylinders are maximal in $Y\setminus\Sigma$. We call the marked
surface \textbf{balanced} if all maximal cylinders have the same
height in every periodic direction; this property is
$\GL$-invariant.

\subsection{The canonical torus covering}
\begin{lemma}\label{lem:relative-torus}
Suppose that $(Y,\omega,\Sigma)$ is balanced, $Z(\omega)\subset\Sigma\ne\emptyset$ and $\Lambda_{\mathrm{abs}}$ is a rank-two lattice. Then
$\Lambda_{\mathrm{rel}}$ is a rank-two lattice, and integration from
a point of $\Sigma$ defines a translation covering
\[
 Y\longrightarrow\C/\Lambda_{\mathrm{rel}}
\]
with all of $\Sigma$ over the origin. Its degree is the $\GL$-invariant
positive integer
\begin{equation}\label{eq:degree}
 N(Y,\omega,\Sigma)=\frac{\Area(Y,\omega)}{\covol\Lambda_{\mathrm{rel}}}.
\end{equation}
After normalizing $\Lambda_{\mathrm{rel}}$ to $\Z+i\Z$, the surface
is tiled by exactly $N(Y,\omega,\Sigma)$ unit squares, with every
point of $\Sigma$ at a vertex.
\end{lemma}

\begin{proof}
Choose $p_0\in\Sigma$ and $A\in\GL$ with
$A\Lambda_{\mathrm{abs}}=\Z+i\Z$. Set
$\omega_1=A_*\omega$ and
$\Lambda_{\mathrm{rel}}^{(1)}=A\Lambda_{\mathrm{rel}}$.
Integration modulo $\Z+i\Z$ gives a finite translation
covering
\[
 F:Y\longrightarrow \C/(\Z+i\Z),\qquad
 F(p)=\int_{p_0}^p\omega_1\pmod{\Z+i\Z}.
\]
The horizontal and vertical directions on the square
torus are completely periodic, and their lifts induce
the corresponding marked cylinder decompositions on
$Y$ (cf.~\cite[\S~2.3]{HubertLelievre}).
By balance, all horizontal cylinders have a common
height $h_{\mathrm h}>0$, and all vertical cylinders
have a common height $h_{\mathrm v}>0$.

Let $\mathcal G_{\mathrm h}$ be the horizontal critical
graph, and let $C_1,\ldots,C_r$ be the complementary
horizontal cylinders. For each $C_j$, choose a point
$b_j$ on one boundary component of its metric
completion. Define
\begin{align}\label{eq:uh}
u_{\mathrm h}:Y\longrightarrow
\R/h_{\mathrm h}\Z,
\qquad
u_{\mathrm h}(p)=
\begin{cases}
\displaystyle
\operatorname{Im}\int_{b_j}^{p}\omega_1
\pmod{h_{\mathrm h}},
& p\in C_j,\\[6pt]
0,&p\in\mathcal G_{\mathrm h}.
\end{cases}
\end{align}
The imaginary part of the integral is independent of
the path within the cylinder, since its core periods
are real. As each cylinder has height $h_{\mathrm h}$,
the two boundary components both map to zero.
Thus $u_{\mathrm h}$ is continuous on $Y$ and satisfies
\[
du_{\mathrm h}=\operatorname{Im}\omega_1
\quad\text{on }Y\setminus\mathcal G_{\mathrm h},
\qquad
u_{\mathrm h}|_\Sigma=0.
\]
Consequently, for every $\sigma\in\Sigma$,
\[
\operatorname{Im}\int_{p_0}^{\sigma}\omega_1
\in h_{\mathrm h}\Z.
\]

Similarly, let $\mathcal G_{\mathrm v}$ be the vertical
critical graph, with complementary cylinders
$C'_1,\ldots,C'_{r'}$. Choosing a point $b'_j$ on
one boundary component of each $C'_j$, define
\begin{align}\label{eq:uv}
u_{\mathrm v}:Y\longrightarrow
\R/h_{\mathrm v}\Z,
\qquad
u_{\mathrm v}(p)=
\begin{cases}
\displaystyle
\operatorname{Re}\int_{b'_j}^{p}\omega_1
\pmod{h_{\mathrm v}},
& p\in C'_j,\\[6pt]
0,&p\in\mathcal G_{\mathrm v}.
\end{cases}
\end{align}
The same argument gives
\[
du_{\mathrm v}=\operatorname{Re}\omega_1
\quad\text{on }Y\setminus\mathcal G_{\mathrm v},
\qquad
u_{\mathrm v}|_\Sigma=0,
\]
and hence
\[
\operatorname{Re}\int_{p_0}^{\sigma}\omega_1
\in h_{\mathrm v}\Z
\qquad(\sigma\in\Sigma).
\]
Therefore
\[
\Z+i\Z\subseteq\Lambda_{\mathrm{rel}}^{(1)}
\subseteq h_{\mathrm v}\Z\oplus i\,h_{\mathrm h}\Z,
\]
so $\Lambda_{\mathrm{rel}}$ is a rank-two lattice.

The relative-period construction
\cite[Lemma~2.1]{HubertLelievre} now gives a finite translation
covering $G:Y\to\C/\Lambda_{\mathrm{rel}}^{(1)}$ with
$G(\Sigma)=\{0\}$. As $G^*dz=\omega_1$,
\[
 \deg G=
 \frac{\Area(Y,\omega_1)}{\covol\Lambda_{\mathrm{rel}}^{(1)}}
 =\frac{\Area(Y,\omega)}{\covol\Lambda_{\mathrm{rel}}}.
\]
This ratio is $\GL$-invariant. Normalizing the relative lattice
to $\Z+i\Z$ yields exactly $\deg G$ unit squares, with all
points of $\Sigma$ at vertices.
\end{proof}

We say that a vector $u\in\Lambda_{\mathrm{rel}}$ is \textbf{primitive}
if $u/k\notin\Lambda_{\mathrm{rel}}$ for every integer $k\geq2$,
equivalently $u$ generates $\R u\cap\Lambda_{\mathrm{rel}}$. The
direction of $u$ is completely periodic by the torus covering, so it
decomposes $Y\setminus\Sigma$ into marked maximal cylinders; by
balance these have a common height.
\begin{lemma}\label{lem:height}
Under the hypotheses of Lemma~\ref{lem:relative-torus}, set
$A_{\mathrm{rel}}=\covol\Lambda_{\mathrm{rel}}$. For a primitive vector
$u\in\Lambda_{\mathrm{rel}}$, the
common cylinder height is
\begin{equation}\label{eq:exact-height}
 h=\frac{A_{\mathrm{rel}}}{|u|}.
\end{equation}
If cylinders in two nonparallel directions have heights $h_1,h_2$,
then
\begin{equation}\label{eq:determinant}
 A_{\mathrm{rel}}\geq h_1h_2.
\end{equation}
\end{lemma}

\begin{proof}

Put $\delta:=A_{\mathrm{rel}}/|u|$ and for any $z\in\C$ consider the
transverse coordinate
\[
 y(z):=\frac{\det(u,z)}{|u|}.
\]
Since $u$ is primitive in $\Lambda_{\mathrm{rel}}$,
we have $y(\Lambda_{\mathrm{rel}})=\delta\Z$.

We first prove $h\in\delta\Z$. Let $C$ be a marked maximal cylinder in this direction.
Let $\partial_-C$ and $\partial_+C$ denote the lower and
upper boundary components of the metric completion of $C$,
respectively. Choose points
\[
\xi_1\in \partial_-C\setminus\Sigma,
\qquad
\xi_2\in \partial_+C\setminus\Sigma,
\]
so that the segment $\overline{\xi_1\xi_2}$ is transverse to the
core curves of $C$. Since each boundary component meets
$\Sigma$, choose points
\[
\sigma_1\in \partial_-C\cap\Sigma,
\qquad
\sigma_2\in \partial_+C\cap\Sigma.
\]
Joining three lines $\overline{\sigma_1\xi_1}$, $\overline{\xi_1\xi_2}$, and
$\overline{\xi_2\sigma_2}$ gives a path $\gamma$ from
$\sigma_1$ to $\sigma_2$; see
Figure~\ref{fig:cylinder-gamma}.
\begin{figure}[ht]
\centering
\begin{tikzpicture}[scale=1.0,>=Latex]

% rectangle for the planar model of the cylinder
\draw[thick] (0,0) rectangle (8,3);

% labels for dimensions
\draw[<->] (0,-0.45) -- (8,-0.45);
\node at (4,-0.75) {$L$};

\draw[<->] (8.45,0) -- (8.45,3);
\node[right] at (8.45,1.5) {$h$};

% boundary components
\node[left] at (0,3) {$\partial_+ C$};
\node[left] at (0,0) {$\partial_- C$};

% distinguished points Sigma on the boundary
\fill (1.1,0) circle (1.7pt);
\node[below] at (1.1,0) {$\sigma_1$};

\fill (6.9,3) circle (1.7pt);
\node[above] at (6.9,3) {$\sigma_2$};

% auxiliary points xi on the boundary where the vertical arc meets
\fill (4.1,0) circle (1.7pt);
\node[below] at (4.1,0) {$\xi_1$};

\fill (4.1,3) circle (1.7pt);
\node[above] at (4.1,3) {$\xi_2$};

% completed relative path gamma
\draw[red,very thick,->] (1.1,0) -- (4.1,0);
\draw[red,very thick,->] (4.1,0) -- (4.1,3);
\draw[red,very thick,->] (4.1,3) -- (6.9,3);

\node[red] at (4.5,1.7) {$\gamma$};

% indicate the periodic direction u
\draw[->,thick] (1.4,1.5) -- (3.4,1.5);
\node[above] at (2.4,1.5) {$u$};

% optional core curve
\draw[dashed] (0.6,1.0) -- (7.8,1.0);
\node[below] at (0.7,1.0) {$\gamma_{\mathrm{core}}$};

\end{tikzpicture}
\caption{A marked maximal cylinder and $\gamma$.}
\label{fig:cylinder-gamma}
\end{figure}
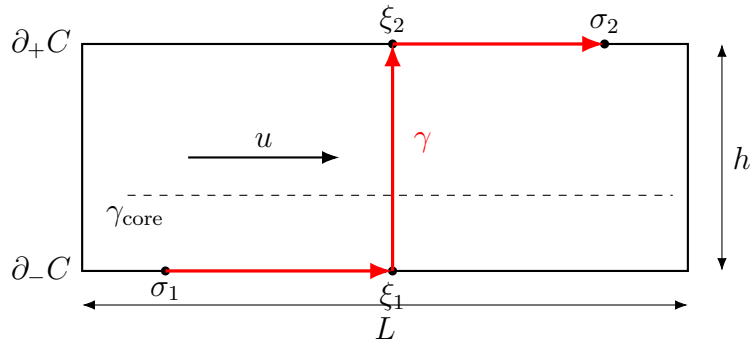

The two added boundary arcs $\overline{\sigma_1\xi_1}$ and
$\overline{\xi_2\sigma_2}$ are parallel to
$u$ and hence contribute no transverse displacement.
Thus, writing $v:=\int_\gamma\omega\in\Lambda_{\mathrm{rel}}$,
we have
\[
h|u|=|\det(u,v)|\in A_{\mathrm{rel}}\Z,
\]
which gives $h\in\delta\Z$.

Conversely, by balance, all marked maximal cylinders
in this direction have the same height $h$. As in the proof of Lemma~\ref{lem:relative-torus},
their transverse height coordinates modulo $h$ define
a continuous circle-valued map
\[
f:Y\longrightarrow\R/h\Z
\]
which vanishes on
$\Sigma$. If $u=|u|e^{i\vartheta}$, then
$df=\operatorname{Im}(e^{-i\vartheta}\omega)$
on the cylinder interiors. Consequently, for any
relative path $\gamma$ with endpoints
$p_-,p_+\in\Sigma$,
\[
\operatorname{Im}\left(
e^{-i\vartheta}\int_\gamma\omega
\right)
\equiv f(p_+)-f(p_-)\equiv0\pmod{h\Z}.
\]
The same holds for closed paths, so
\[
\operatorname{Im}
\left(e^{-i\vartheta}\Lambda_{\mathrm{rel}}\right)
\subseteq h\Z.
\]
Since $u$ is primitive in $\Lambda_{\mathrm{rel}}$,
the group on the left is $\delta\Z$.
Thus $\delta\in h\Z$. Together with
$h\in\delta\Z$ and $h,\delta>0$, this yields
$h=\delta$, proving \eqref{eq:exact-height}.

For two nonparallel directions, choose primitive
lattice vectors $u,v$ in those directions.
Their determinant is a nonzero integral multiple
of $A_{\mathrm{rel}}$, so
\[
 A_{\mathrm{rel}}
 \leq|\det(u,v)|
 \leq|u||v|
 =\frac{A_{\mathrm{rel}}^2}{h_1h_2}.
\]
Dividing by $A_{\mathrm{rel}}>0$ proves
\eqref{eq:determinant}.
\end{proof}

\subsection{Arithmeticity}\label{sec:arithmetic}

To apply the canonical torus covering to a pointed quadratic surface,
we first compare its cylinders with those on its orientation cover.

\begin{lemma}\label{lem:orientation-cylinders}
Let $q\in H^0(X,K_X^2(P))$ be nonsquare, with canonical orientation
cover $p:Y\to X$ and $p^*q=\omega^2$. Assume $\Sigma=Z(\omega)\cup p^{-1}(P)$. Each marked maximal cylinder on
$(Y,\omega,\Sigma)$ is exchanged
with a distinct cylinder by the deck involution and projects
isometrically to a marked maximal cylinder on $(X,P,q)$. In particular,
if $(X,P,q)$ is balanced, then so is $(Y,\omega,\Sigma)$.
\end{lemma}

\begin{proof}
Let $\sigma$ be the deck involution, so that
$\sigma^*\omega=-\omega$. Every fixed point of
$\sigma$ belongs to $\Sigma$. If a marked maximal
cylinder $C$ were invariant under $\sigma$, then
in flat coordinates
\[
C\cong(\R/L\Z)\times(0,h),
\]
the involution would have the form
\[
(x,y)\longmapsto(-x+c,h-y),
\]
which has fixed points in $C$, a contradiction.
Thus $\sigma$ exchanges $C$ with a distinct
cylinder. Since $p$ is a local isometry away
from $\Sigma$, its restriction to $C$ is an
isometry onto a marked maximal cylinder
downstairs. Conversely, every marked maximal
cylinder downstairs lifts to such a pair.
The equality of heights follows.
\end{proof}

\begin{proof}[Proof of Theorem~\ref{thm:structure}\ref{it:arithmetic}]
Let $(X,P,q)$ generate the Teichm\"uller disk of $\mathcal C$,
as in Theorem~\ref{thm:structure}\ref{it:disk}.
Write $p:Y\to X$ for the canonical orientation cover and
$p^*q=w^2$; in the square case take $Y=X$ and
$p=\operatorname{id}$. Keep the distinguished set
$\Sigma=Z(w)\cup p^{-1}(P)$.

By Lemma~\ref{lem:absolute-periods}, the absolute periods of
$w$ form a rank-two lattice. In the nonsquare case,
Theorem~\ref{thm:structure}\ref{it:heights} and
Lemma~\ref{lem:orientation-cylinders} show that
$(Y,w,\Sigma)$ is balanced. In the square case, balance is
exactly the equal-height assertion on the original pointed surface.

The set $\Sigma$ is nonempty. Indeed, if $w$ has no zeros,
then $Y$ has genus one. Riemann--Hurwitz for $p$ gives $g\leq1$,
and stability forces $n\geq1$, so $p^{-1}(P)\ne\varnothing$. Lemma~\ref{lem:relative-torus} now gives a translation covering
\[
 F:Y\longrightarrow\C/\Lambda_{\mathrm{rel}}
\]
carrying every point of $\Sigma$ to the origin. Since
$F^*dz=w$, its ramification points are zeros of $w$,
all of which belong to $\Sigma$. Thus $F$ is ramified only over
the origin. Hence $\mathcal C$ is arithmetic. The associated
Veech group is also arithmetic by \cite{GutkinJudge}.
This proves Theorem~\ref{thm:structure}\ref{it:arithmetic}
and completes the proof of Theorem~\ref{thm:structure}.
\end{proof}

\begin{corollary}\label{coro:WPcurve}
Every closed algebraic WP totally geodesic curve
$\mathcal C\subset\M_{g,n}$ arises from a covering
construction of $\M_{0,4}$, possibly after a finite
base change.
\end{corollary}

\begin{proof}
Keep the notation of the proof of
Theorem~\ref{thm:structure}\ref{it:arithmetic}.
By Lemma~\ref{lem:relative-torus}, there is a
translation covering
\[
 F:Y\longrightarrow \C/\Lambda_{\mathrm{rel}},
 \qquad F^*dz=\omega,\qquad F(\Sigma)=\{0\},
\]
ramified only over the origin.
Let $T=\mathbb C/\Lambda_{\mathrm{rel}}$ and
\[
\epsilon:\C/\Lambda_{\mathrm{rel}}\longrightarrow T/\{\pm1\}\simeq\mathbb P^1
\]
be the quotient map.
Thus $\epsilon(T[2])$ consists of four points. The quadratic
differential $dz^2$ descends to a quadratic differential
$q_0$ on $\mathbb P^1$ with simple poles exactly
at $\epsilon(T[2])$.

When $q$ is nonsquare, the deck involution
satisfies $\sigma^*\omega=-\omega$. Since
$\sigma(\Sigma)=\Sigma$ and $F(\Sigma)=\{0\}$,
we have $F\circ\sigma=-F$. Consequently
$\epsilon\circ F$ descends through $p:Y\to X$
to a finite covering
\[
 h:X\longrightarrow\mathbb P^1,
 \qquad h\circ p=\epsilon\circ F.
\]
When $q$ is a square, simply take $h=\epsilon\circ F$.
In either case,
\[
 q=h^*q_0.
\]

The branch values of $h$ lie in $\epsilon(T[2])$, and
$h(P)\subseteq \epsilon(T[2])$. At a point $x\in h^{-1}(\epsilon(T[2]))$
with ramification index $e_x$, the pullback
formula gives $\operatorname{ord}_x(q)=e_x-2.$
Since $q$ has poles only at $P$, it follows that
\[
 h^{-1}(\epsilon(T[2]))=P\cup\operatorname{Ram}(h).
\]
Thus $h$ defines a covering construction from
$\mathcal T_{0,4}$ to $\mathcal T_{g,n}$, with
unmarked ramification points forgotten.
\end{proof}

\subsection{Bounding the degree}
In this subsection only, take $G\geq2$ and assume $\Sigma=Z(\omega)$. Let $\mathbf k=(k_1,\ldots,k_s)$ be a partition of $2G-2$;
$\mathcal H_1(\mathbf k)$ denotes the stratum of area-one Abelian
differentials whose zeros have orders $k_1,\ldots,k_s$.

We use the following consequence of \cite[Theorem~2.3]{EMM}: the
canonical ergodic $\SL$-invariant probability measures of closed
$\SL$-orbits in an Abelian-differential stratum $\mathcal H_1(\mathbf k)$ are sequentially compact in the weak-star topology, with limits
supported on affine invariant submanifolds of the same stratum; in
particular the limit has total mass one and no mass escapes through
the Deligne--Mumford boundary. Square-tiled surfaces have closed
orbits \cite{GutkinJudge}.

\begin{theorem}\label{thm:flat}
For each integer $G\geq2$, there are only finitely many $\GL$-orbits
of balanced translation surfaces of genus $G$ whose absolute periods
form a rank-two lattice. 
\end{theorem}
\begin{proof}

\emph{Bounding the canonical degree.}
Suppose that the degrees in \eqref{eq:degree} were unbounded,
and choose a sequence $(Y_j,\omega_j)$ with degrees
$N_j\to\infty$. Normalize their areas to one, so the
relative-lattice covolumes are $a_j=N_j^{-1}\to0$.
After passing to a subsequence, the surfaces lie in a fixed
connected component of an area-one Abelian stratum
$\mathcal H_1(\mathbf k)$.
By Lemma~\ref{lem:relative-torus}, they are square-tiled
up to $\GL$, so their $\SL$-orbits
\[
\mathcal U_j:=\SL\cdot(Y_j,\omega_j)
\]
are closed \cite{GutkinJudge}.
Let $\mu_j$ be the invariant probability measure on
$\mathcal U_j$. By \cite[Theorem~2.3]{EMM}, after passing
to a subsequence, $\mu_j$ converges weakly to an affine
probability measure $\mu$ on $\mathcal H_1(\mathbf k)$.

Choose $(Y_\infty,\omega_\infty)\in\operatorname{supp}\mu$.
By \cite{MasurCylinders}, it contains cylinders in two
nonparallel directions. Cylinders persist under sufficiently
small deformations, with their directions and heights varying
continuously \cite[\S5.2]{WrightSurvey}.
Thus there exist an open neighborhood $V$ of
$(Y_\infty,\omega_\infty)$ in $\mathcal H_1(\mathbf k)$
and a constant $c>0$ such that every surface in $V$
contains cylinders in two nonparallel directions,
each of height at least $c$.

Since $\mu(V)>0$, weak convergence gives
$\mu_j(V)>0$ for all sufficiently large $j$.
Choose $(\widehat Y_j,\widehat\omega_j)\in
\mathcal U_j\cap V$.
This surface remains balanced, and its relative-lattice
covolume is $a_j$, since both are $\SL$-invariant.
Let $H_{1,j},H_{2,j}\geq c$ be the heights of the
corresponding maximal cylinders.
By \eqref{eq:determinant},
\[
a_j\geq H_{1,j}H_{2,j}\geq c^2,
\]
contradicting $a_j\to0$.
Hence the canonical degree is bounded in terms of $G$.

\emph{Counting the coverings.}
Normalize each relative lattice to $\Z+i\Z$.
Lemma~\ref{lem:relative-torus} gives a connected square-tiled
surface of bounded degree $N$. Such a surface is encoded by
a transitive pair of permutations in $S_N\times S_N$,
specifying the identifications of the right and upper edges.
There are only finitely many such pairs for bounded $N$.
Hence there are only finitely many surfaces in this
normalization and, therefore, finitely many $\GL$-orbits.
\end{proof}

\subsection{Finiteness of arithmetic Teichm\"uller curves
arising from balanced surfaces}
We use the following lemma to relax the condition
$\Sigma=Z(\omega)$ in Theorem~\ref{thm:flat} and deduce finiteness Theorem~\ref{thm:marked-flat}.
\begin{lemma}\label{lem:mark-cover}
Let $(Y,\omega,\Sigma)$ have genus $G\geq1$, and let
$R:=\Sigma\setminus Z(\omega)$ have cardinality $r\geq1$. Then there is a connected regular translation covering
\[
 F:(Z,\eta)\longrightarrow(Y,\omega),\qquad \eta=F^*\omega,
\]
of degree six, with deck group $S_3$, branched exactly over $R$
and with ramification index three at every point above $R$.
Its genus $G_Z$ is
\begin{equation}\label{eq:mark-cover-genus}
 G_Z=6G-5+2r\geq3,
\end{equation}
and $F^{-1}(\Sigma)=Z(\eta)$. If the marked surface $(Y,\omega,\Sigma)$ is
balanced and has lattice absolute periods, then $(Z,\eta)$ is
balanced with respect to $Z(\eta)$ and has lattice absolute periods.
\end{lemma}

\begin{proof}
Write the fundamental group of the punctured surface as
\[
 \pi_1(Y\setminus R)=
 \left\langle a_i,b_i,c_j\ \middle|\
       \prod_{i=1}^G[a_i,b_i]\prod_{j=1}^r c_j=1\right\rangle.
\]
Let $\varpi=(123)\in S_3$, and send every $c_j$ to $\varpi$.
We distinguish two cases:
\begin{enumerate}
\item[(1)] If $3\mid r$, send $a_1$ to $(12)$ and all
$b_i$ and remaining handle generators to the identity.

\item[(2)] If $3\nmid r$, send $a_1$ and $b_1$ to two
distinct transpositions whose commutator is
$\varpi^{-r}$, and all remaining handle generators
to the identity.
\end{enumerate}
In either case, we have a surjective homomorphism from $\pi_1(Y\setminus R)$ to $S_3$. By the Riemann existence theorem, the associated unramified covering of
$Y\setminus R$ extends to a holomorphic branched
covering $F:Z\to Y$.

Since the cover has degree six and each peripheral
monodromy has order three, every branch value has
exactly two preimages, each of ramification index three. Riemann--Hurwitz gives
$2G_Z-2=6(2G-2)+4r$, proving \eqref{eq:mark-cover-genus}.
Every regular marked point lifts to a zero of order two, all old
zeros lift to zeros, and there are no other zeros; hence
$F^{-1}(\Sigma)=Z(\eta)$.

A cylinder downstairs lifts to cylinders with the same height,
possibly with circumferences multiplied by an integer. Conversely,
every cylinder upstairs is a component of the preimage of a marked
maximal cylinder downstairs: its closed regular trajectories project
to closed regular trajectories, and its boundary cannot cross
$F^{-1}(\Sigma)=Z(\eta)$. Periodicity and equal heights therefore
pass to the cover. Finally, transfer of integral cycles gives
\[
 6\Lambda_{\mathrm{abs}}(Y,\omega)
 \subset\Lambda_{\mathrm{abs}}(Z,\eta)
 \subset\Lambda_{\mathrm{abs}}(Y,\omega).
\]
Thus the absolute periods upstairs are also a rank-two lattice.
\end{proof}

\begin{theorem}\label{thm:marked-flat}
Fix $G\geq1$ and $r_0\geq0$. There are only finitely many
$\GL$-orbits of balanced marked translation surfaces
$(Y,\omega,\Sigma)$ of genus $G$ with lattice absolute periods,
where $Z(\omega)\subseteq\Sigma\ne\emptyset$ and $\#(\Sigma\setminus Z(\omega))\leq r_0$. 
\end{theorem}

\begin{proof}
If $\Sigma=Z(\omega)\ne\emptyset$, then $G\geq2$ and
Theorem~\ref{thm:flat} applies. Otherwise,
Lemma~\ref{lem:mark-cover} gives a translation covering
$(Z,\eta)\to(Y,\omega)$ with
$3\leq G_Z\leq6G-5+2r_0$.
By Theorem~\ref{thm:flat}, there are finitely many
possibilities for $(Z,\eta)$ up to $\GL$.
Since $\Aut(Z)$ is finite, only finitely many
deck groups and hence quotients occur.
\end{proof}

\subsection{Finiteness conclusion}

\begin{proof}[Proof of Theorem~\ref{thm:finiteness}]
Let $\mathcal C\subset\M_{g,n}$ be an algebraic WP totally
geodesic curve, generated by a pointed quadratic differential
$(X,P,q)$ as in Theorem~\ref{thm:structure}.
Let $p:Y\to X$ be its orientation cover, with
$p^*q=w^2$ and distinguished set $\Sigma$ as in
\eqref{eq:distinguished}; in the square case take
$p=\operatorname{id}$.
The proof of
Theorem~\ref{thm:structure}\ref{it:arithmetic}
shows that $(Y,w,\Sigma)$ is balanced, has lattice
absolute periods, and satisfies $\Sigma\ne\varnothing$.
Moreover, $\Sigma$ contains at most $2n$ regular points.

Write $G_Y\geq1$ for the genus of $Y$.
In the nonsquare case, the cover branches over the
odd-order zeros and simple poles of $q$, of which there
are at most $4g-4+2n$. Riemann--Hurwitz therefore gives
\[
G_Y\leq4g-3+n.
\]
In the square case, $G_Y=g$.
Thus Theorem~\ref{thm:marked-flat}, with $r_0=2n$,
gives finitely many possibilities for
$(Y,w,\Sigma)$ up to $\GL$.

Fix a representative of each such orbit.
In the nonsquare case, transporting the orientation
involution to this representative by a real affine map
preserves its derivative $-I$. It is therefore a
holomorphic involution $\sigma$ satisfying
\[
\sigma^*w=-w,
\qquad
\sigma(\Sigma)=\Sigma.
\]
Since $\Sigma\ne\varnothing$, the group
$\Aut(Y,\Sigma)$ is finite, including when $G_Y=1$.
Hence only finitely many such involutions occur.

For each admissible involution, the quotient
$X=Y/\langle\sigma\rangle$ and the descended quadratic
differential $q$ are determined. The ordered markings
are chosen from the finite set $p(\Sigma)$, subject to
$\Sigma=Z(w)\cup p^{-1}(P)$, so there are only
finitely many choices.
In the square case, $X=Y$ and $q=w^2$ are already
determined, and the same finite marking argument applies.
We obtain finitely many pointed quadratic $\GL$-orbits,
hence finitely many possible moduli images $\mathcal C$.
\end{proof}

\appendix
\section{Algebraic extension of the square-root splitting}
\label{app:algebraic-splitting}

We prove here the extension statement used in
Lemma~\ref{lem:factor}\ref{it:existence}. 

\begin{lemma}
\label{lem:algebraic-splitting}
In the setting of the proof of Lemma~\ref{lem:factor}, the maps
$\lambda_V:\mathcal L|_V\to E|_V$ and
$r_V:E|_V\to\mathcal L|_V$ defined there extend uniquely to
morphisms of vector bundles
\[
 \lambda:\mathcal L\longrightarrow E,\qquad
 r:E\longrightarrow\mathcal L
\]
on $B'$. They satisfy $r\lambda=\operatorname{id}_{\mathcal L}$.
\end{lemma}

\begin{proof}
Keep the notation of Section~\ref{sec:cover}; in particular,
$\kappa'\beta'=\operatorname{id}_{\mathcal L^{\otimes2}}$.
In this proof, write $q'$ for the pullback of $q_1$ to
$\mathcal X'$, equivalently the section induced by $\beta'$:
\[
 q'\in H^0\!\left(\mathcal X',
 \omega_{\mathcal X'/B'}^{\otimes2}(S')\otimes
 f'^*\mathcal L^{-2}\right).
\]
The finite flat double cover
$\varpi:\mathcal Y_1\to\mathcal X'$ is the base change
of the cover defined in Section~\ref{sec:cover} by
$(\mathcal M_1,q_1)$; see \eqref{eq:diagram}.

\emph{Extension of the square root.}
The section $\lambda_V$ constructed in Lemma~\ref{lem:factor} is induced
by the square root of $q'$ on $V$. It therefore determines a rational section
\[
 \alpha\in H^0_{\mathrm{rat}}\!\left(\mathcal Y,
 \omega_{\mathcal Y/B'}\otimes g^*\mathcal L^{-1}\right),
 \qquad \alpha^{\otimes2}=p^\#q',
\]
where $p^\#q'$ denotes the rational relative quadratic
differential obtained by pulling back $q'$ along the
induced morphism of generic fibers
$p_\eta:\mathcal Y_\eta\to\mathcal X'_\eta$.
By the canonical-cover construction
\cite[\S2.1, Proposition~2.4]{BCGGM}, this root $\alpha$ is regular on
the normalized smooth generic covering curve, including above
marked simple poles. Hence $\alpha$ has no horizontal poles.

Let $C$ be a vertical prime divisor of $\mathcal Y$, with generic
point $\eta_C$. Then $A:=\mathcal O_{\mathcal Y,\eta_C}$ is a DVR,
and the local parameter $s$ of $B'$ is its uniformizer because $g$
is semistable. By properness of $f'$ and
\cite[Lemma~29.43.5, Tag~0BX7]{Stacks}, the rational map
$p:\mathcal Y\dashrightarrow\mathcal X'$ is a morphism near
$\eta_C$. Thus regular functions near $p(\eta_C)$ pull back to
$A$. If $p(\eta_C)$ lies in the relative smooth unmarked locus,
regularity of $p^\#q'$ follows from the usual pullback of relative
differentials. There remain only two local computations.

At a marked section $z=0$, trivialize $\mathcal L$ and write
$q'=c(dz)^2/z$, where $c$ is regular. In $A$ write
$p^*z=s^m v$ with $m\geq0$ and $v\in A^\times$. Since
$d_{\mathcal Y/B'}s=0$,
\[
 p^\#q'=(p^*c)\,s^m
 \frac{(d_{\mathcal Y/B'}v)^2}{v}
 \quad\text{is regular at }\eta_C.
\]
At a node of $f'$, take $xy=s^e u(s)$, $u(0)\ne0$, and
$q'=c(dx/x)^2$ with $c$ regular, using the standard local
generator of the relative dualizing sheaf. Writing
$p^*x=s^m v$ with $v\in A^\times$ gives
\[
 p^\#q'=(p^*c)
 \left(\frac{d_{\mathcal Y/B'}v}{v}\right)^2
 \quad\text{is regular at }\eta_C.
\]
Thus $p^\#q'$ has no poles along any prime divisor of
$\mathcal Y$. The equality $\alpha^{\otimes2}=p^\#q'$ shows
that the same holds for $\alpha$. Since $\mathcal Y$ is normal
and $\omega_{\mathcal Y/B'}\otimes g^*\mathcal L^{-1}$ is
invertible, the codimension-one criterion
\cite[Tag~031T]{Stacks} makes $\alpha$ a global regular
section. By adjunction it defines
\[
 \lambda:\mathcal L\longrightarrow g_*\omega_{\mathcal Y/B'}.
\]
In the nonsquare case, its generic anti-invariance under the deck
involution implies that $\lambda$ takes values in $E$; in the
square case $E=g_*\omega_{\mathcal Y/B'}$ by definition.

\emph{The nonsquare case.}
One may refer to \eqref{eq:diagram} for the covering convention. Finite duality
\cite[Remark~48.20.12, Tag~0AU5]{Stacks} and the defining quadratic algebra $\mathcal A_1$ give
\[
 (f'\circ\varpi)_*\omega_{\mathcal Y_1/B'}\simeq
 f'_*\omega_{\mathcal X'/B'}\ \oplus\
 \bigl(f'_*\omega_{\mathcal X'/B'}^{\otimes2}(S')
 \otimes\mathcal L^{-1}\bigr).
\]
Let $\operatorname{pr}_2$ denote projection onto the second
summand. Resolve the graph of the rational map
$\mathcal Y\dashrightarrow\mathcal Y_1$ equivariantly
to obtain proper birational morphisms
\[
 \mathcal Y\xleftarrow{\ \rho\ }\widetilde{\mathcal Y}
 \xrightarrow{\ \psi\ }\mathcal Y_1.
\]
Since the normal semistable surface $\mathcal Y$ has at
worst rational double points, the birational trace
\[
 \operatorname{Tr}_{\rho}:
 \rho_*\omega_{\widetilde{\mathcal Y}/B'}
 \xrightarrow{\ \sim\ }\omega_{\mathcal Y/B'}
\]
is an isomorphism
\cite[Proposition~4.6]{Cais}.
On the other hand, proper duality gives a trace map
\[
 \operatorname{Tr}_{\psi}:
 \psi_*\omega_{\widetilde{\mathcal Y}/B'}
 \longrightarrow\omega_{\mathcal Y_1/B'};
\]
see \cite[Example~48.20.11, Tag~0AU5]{Stacks}. Since $g\circ\rho=(f'\circ\varpi)\circ\psi$,
pushing these maps forward to $B'$ yields
\[
g_*\omega_{\mathcal Y/B'}
\xleftarrow[\sim]{\ g_*\operatorname{Tr}_{\rho}\ }
(g\circ\rho)_*\omega_{\widetilde{\mathcal Y}/B'}=(f'\circ\varpi\circ\psi)_*
\omega_{\widetilde{\mathcal Y}/B'}\xrightarrow{\ (f'\circ\varpi)_*
\operatorname{Tr}_{\psi}\ }
(f'\circ\varpi)_*\omega_{\mathcal Y_1/B'}.
\]
We therefore obtain a regular morphism
\[
 T:g_*\omega_{\mathcal Y/B'}
 \longrightarrow
 (f'\circ\varpi)_*\omega_{\mathcal Y_1/B'}.
\]
By the equivariance of the resolution and the
naturality of the trace maps, $T$ is
equivariant under the deck involution.
Consequently,
\begin{align}\label{eq:P}
 P:=\frac12\operatorname{pr}_2\circ T\big|_E:
 E\longrightarrow
 f'_*\omega_{\mathcal X'/B'}^{\otimes2}(S')
 \otimes\mathcal L^{-1}
\end{align}
is regular on all of $B'$.

On $V$, we identify $P|_V$ with the normalized trace in
\eqref{eq:retraction}. It suffices to work over the dense
open subset of $\mathcal X'_V$ away from $S'$ and the zeros
of $q'$, where $p_V$ is \'etale.
Choose a local generator $a$ of $\mathcal L|_V$ and a
relative coordinate $z$. With pullbacks understood, write
\[
 \beta'(a^{\otimes2})=c(dz)^2,\qquad
 w_a=\alpha\,g^*a=u\,dz,\qquad u^2=c.
\]
Here $c$ and $u$ are units. For $\xi\in E_V$, anti-invariance
therefore gives $\xi=u\,b\,dz$, with $b$ regular on the
chosen chart of $\mathcal X'_V$.
Under the finite-duality identification above, using the
frame $dz\otimes a^{-1}$, we obtain
\[
 \begin{aligned}
 \operatorname{pr}_2(T(\xi))
 &=\operatorname{Tr}_{p_V}(u\xi)
       \otimes(dz\otimes a^{-1})\\
 &=2cb(dz)^2\otimes a^{-1}\\
 &=\operatorname{Tr}_{p_V}(w_a\xi)\otimes a^{-1}.
 \end{aligned}
\]
By \eqref{eq:P}, this yields
\[
 P(\xi)=\frac12\operatorname{Tr}_{p_V}(w_a\xi)
 \otimes a^{-1}.
\]
Both sides define regular bundle maps on $V$, the right-hand
side by the descent calculation in Section~\ref{sec:cover}.
Thus the identity holds throughout $V$. Consequently,
\[
 r:=(\kappa'\otimes\operatorname{id}_{\mathcal L^{-1}})
 \circ P:E\longrightarrow\mathcal L
\]
satisfies
\[
 a\otimes r(\xi)
 =\kappa'\left(\frac12\operatorname{Tr}_{p_V}(w_a\xi)\right),
\]
so $r|_V=r_V$ by \eqref{eq:retraction}, since $\tau=2$.

\emph{The square case.}
Here $\mathcal Y=\mathcal X'$, $g=f'$, $p=\operatorname{id}$
and $E=f'_*\omega_{\mathcal X'/B'}$.
The extended root $\alpha$ is a regular section of
$\omega_{\mathcal X'/B'}\otimes f'^*\mathcal L^{-1}$.
By the projection formula, multiplication by $\alpha$,
followed by the natural inclusion
$\omega_{\mathcal X'/B'}^{\otimes2}
\hookrightarrow\omega_{\mathcal X'/B'}^{\otimes2}(S')$,
defines
\[
 P:E\longrightarrow
 f'_*\omega_{\mathcal X'/B'}^{\otimes2}(S')
 \otimes\mathcal L^{-1},\qquad P(\xi)=\alpha\xi.
\]
Define $r=(\kappa'\otimes\operatorname{id}_{\mathcal L^{-1}})
\circ P$. Since $\alpha\,f'^*a=w_a$ on $V$, this also
restricts to $r_V$, with $\tau=1$.

Both $\lambda$ and $r$ are now regular on $B'$.
Their restrictions satisfy
$r_V\lambda_V=\operatorname{id}_{\mathcal L|_V}$;
hence $r\lambda=\operatorname{id}_{\mathcal L}$ on $B'$
by density.
\end{proof}

\end{document}